\documentclass[11pt]{amsart}
\usepackage{amsmath}
\usepackage{amsfonts}
\usepackage{amsthm}
\usepackage{amssymb}
\usepackage{mathrsfs}
\usepackage{latexsym}
\usepackage{verbatim}
\usepackage{a4wide}
\usepackage{enumerate}
\usepackage{lipsum}
\usepackage{hyperref}
\usepackage[all]{xy}
\usepackage{array}
\newcolumntype{L}{>{$}l<{$}}

\numberwithin{equation}{section}

\renewcommand{\thetheoremName}

\newcommand{\IC}{\mathbb{C}}

\newcommand{\IT}{\mathbb{T}}
\newcommand{\IN}{\mathbb{N}}
\newcommand{\IG}{\mathbb{G}}
\newcommand{\IP}{\mathbb{P}}
\newcommand{\IQ}{\mathbb{Q}}
\newcommand{\IR}{\mathbb{R}}
\newcommand{\IZ}{\mathbb{Z}}

\newcommand{\calC}{\mathcal{C}}

\newcommand{\calH}{\mathcal{H}}

\newcommand{\calN}{\mathcal{N}}
\newcommand{\calO}{\mathcal{O}}

\newcommand{\calT}{\mathcal{T}}

\newcommand{\ib}{\mathfrak{b}}
\newcommand{\il}{\mathfrak{l}}

\newcommand{\ip}{\mathfrak{p}}

\newcommand{\ig}{\mathfrak{g}}

\def\Ind{\mathrm{Ind}}
\def\PGL{\mathrm{PGL}}

\def\End{\mathrm{End}}

\def\GL{\mathrm{GL}}

\def\Gal{\mathrm{Gal}}

\def\SL{\mathrm{SL}}
\def\Frob{\mathrm{Frob}}
\def\Spec{\mathrm{Spec}}
\def\Spf{\mathrm{Spf}}

\def\Id{\mathrm{Id}}
\def\image{\mathrm{im}}

\def\Sym{\mathrm{Sym}}
\DeclareMathOperator\ord{ord}
\DeclareMathOperator\ch{ch}

\newtheorem{theorem}{Theorem}[section]
\newtheorem*{theorem*}{Theorem}

\newtheorem{question}[theorem]{Question}

\newtheorem{lemma}[theorem]{Lemma}
\newtheorem{prop}[theorem]{Proposition}
\newtheorem{cor}[theorem]{Corollary}
\newtheorem{conj}{Conjecture}

\theoremstyle{definition}

\theoremstyle{remark}

\newtheorem{remark}[theorem]{Remark}
\newtheorem{strategy}[theorem]{Strategy}

\begin{document}

\title{A finiteness result for $p$-adic families of Bianchi modular forms}
\author{Vlad Serban}

\thanks{The author is supported by START-prize Y-966 of the Austrian Science Fund (FWF) under P.I. Harald Grobner}

\address{Vlad Serban, Department of Mathematics, EPFL, Station 8, CH-1015 Lausanne, Switzerland}
\email{vlad.serban@epfl.ch}

\begin{abstract}
We study $p$-adic families of cohomological automorphic forms for $\GL(2)$ over imaginary quadratic fields and prove that families interpolating a Zariski-dense set of classical cuspidal automorphic forms only occur under very restrictive conditions. We show how to computationally determine when this is not the case and establish concrete examples of families interpolating only finitely many Bianchi modular forms. 
\end{abstract}

\bibliographystyle{abbrv}
\maketitle

\section{Introduction}\label{sec:intro}
 The theory of $p$-adic variation has played a central role in much of the progress in the Langlands programme during the last decades, pioneered by the work of H. Hida (e.g., \cite{MR848685, MR868300}) on $p$-adic families of ordinary cuspidal modular forms. Hida subsequently generalized the constructions to automorphic forms on $\GL(2)$ over arbitrary number fields in a series of papers \cite{MR960949, MR1203228, MR1313784}. Over arbitrary number fields $F$, one has to deal with the difficulty that, as soon as $F$ is not totally real, the central objects of Hida's construction, namely the universal ordinary $p$-adic Hecke algebras $\IT$ which interpolate by weight Hecke eigensystems coming from classical automorphic forms, are no longer torsion--free over the respective $p$-adic weights. The lack of suitable Shimura varieties also complicates many constructions. 
 \par 
 In this paper, we consider the situation when $F$ is imaginary quadratic, which was first studied in R. Taylor's thesis \cite{MR2636500}. We examine when a $p$-adic family can interpolate a Zariski-dense set of classical cuspidal automorphic forms, as is the case over totally real fields. This question was already addressed in a paper by F. Calegari and B. Mazur \cite{MR2461903}, which mainly focused on the representation-theoretic side. Provided a $p$-adic transcendence-theoretic conjecture holds, their results imply the set of classical deformations can only be dense in some very specific instances. Calegari and Mazur also gave an illustrative example of a family interpolating finitely many classical automorphic forms \cite[Theorem 1.1]{MR2461903}; however the proof supplied contained an omission. In this paper we rectify this and provide further such examples. This omission was also discussed in D. Loeffler's article \cite{Loeffler:2010ab}, together with an exposition of the analogous question of density of classical points on eigenvarieties for $\GL(1)$ over arbitrary number fields. The existence of ordinary $p$-adic families without a Zariski--dense set of classical points was also studied by A. Ash, D. Pollack and G. Stevens \cite{MR2396124} for $\GL(n)/\IQ$. When $n=3$ they are able to exhibit \cite[Theorem 9.4]{MR2396124} such families. 
 Our work thus fits into the larger framework of examining density of classical points on eigenvarieties and deformation spaces (see e.g., \cite{MR3235551,MR3542488}), and of gaining a better understanding of the situations when these are not expected to be dense.
 \par
The question amounts to understanding the support of the relevant Hecke algebra $\IT$ as a module over the ring of $p$-adic weights $\Lambda$. Over arbitrary number fields, Hida has conjectured \cite[Conjecture 4.3]{MR1313784} that the support has codimension the number of pairs of complex embeddings $r_2$ of $F$, and Hida indeed proves that when $F$ is imaginary quadratic the support has pure codimension one \cite[Theorem 6.2]{MR1313784}. Up to connected components, $\Lambda\cong\calO[[X,Y]]$ is a ring of formal power series over the valuation ring $\calO$ of a finite extension  of $\IQ_p$ and the points on $\Spec(\Lambda)$ that can give rise to Hecke eigensystems of classical automorphic forms constitute a discrete set which is Zariski-dense in the weight space. Points giving rise to such eigensystems are called \emph{classical}. Explicitly, up to some normalization, the primes of $\Lambda$ in question are given by
$$P_{k,\zeta,\xi}:=(X+1-(1+p)^k\zeta,Y+1-(1+p)^k\xi)\cdot\Lambda$$
for any $p$-power roots of unity $\zeta,\xi$ and any positive integer $k$, whereas there is no explicit description of the support of $\IT$ in general. Nevertheless, it turns out that the algebraic situation is quite rigid: building on a $p$-adic version in \cite{Serban:2016aa} of an arithmetic geometry result concerning the density of torsion points on subvarieties of algebraic tori, the so-called multiplicative Manin--Mumford Conjecture, we show that components passing through a dense set of points $P_{k,\zeta,\xi}$ are forced to have a special form and prove the following:
\begin{theorem}\label{thm:support}
Let $F$ be an imaginary quadratic field of class number one with ring of integers $\calO_F$ and let $p$ denote a prime split in $\calO_F$. Let $\varphi$ be a cuspidal automorphic form on $\GL(2)$ over $F$, ordinary at $p$, that contributes to the cohomology of some congruence subgroup $\Gamma\subseteq \SL_2(\calO_F)$. Let $\calH$ denote the ordinary $p$-adic family of cusp forms containing $\varphi$. Then the component of the support of $\calH$ in $\Lambda$ passing through $\varphi$ either:
\begin{enumerate}
\item contains all parallel weights or 
\item contains only finitely many points that give rise to classical automorphic forms or 
\item contains only finitely many points $P_{k,\zeta,\xi}$ of weight $k$ different from the weight $k_\varphi$ of $\varphi$; however
there must exist a positive dimensional formal subtorus $\calT$ of $(\widehat{\IG}_m)^2_{/\IZ_p}$ such that every pair of roots of unity $(\zeta,\xi)\in\mu_{p^\infty}^2$ with $(\zeta-1,\xi-1)\in \calT(\overline{\IQ}_p)$ gives rise to classical points $P_{k_\varphi,\zeta,\xi}$ on the component. In particular, the component contains an infinite number of classical points of fixed weight $k_\varphi$ and varying non-parallel $p$-power nebentypus characters.
\end{enumerate}
\end{theorem}
We remark that the assumption on the class number is not necessary, but to simplify the paper we have not included a detailed proof for arbitrary imaginary quadratic fields and therefore add that assumption. We have however included results that make it clear how our strategy goes through without difficulty for arbitrary imaginary quadratic fields $F$, so that Theorem \ref{thm:support} holds in greater generality, see also Remark \ref{rem:generalh}.
These results are established in Section \ref{sec:two}, while Section \ref{sec:one} provides the necessary background. \par
The first possibility of Theorem \ref{thm:support} indeed occurs, for instance for Hida families which are base-changed from $\IQ$. Still, in the absence of a family of classical points coming from functoriality constructions, one hopes to prove that only finitely many classical points lie on a given family. Our results imply this can be verified, whenever true, with a finite amount of computation. 
In Section \ref{sec:computations} we provide a method for doing so based on computationally eliminating the first and third possibilities in Theorem \ref{thm:support}. We then concretely carry out such computations, which required writing new code to compute with Bianchi modular forms with nebentypus. To that end we also prove a dimension formula for spaces of Eisenstein series with nebentypus (Proposition \ref{prop:eisold}). This approach allows us to come up with the first proven examples of families with finitely many classical points. We show the following (see Theorem \ref{theorem:maincomp}), which in particular closes the gap in the proof of \cite[Theorem 1.1]{MR2461903} and implies a negative answer to a question raised in Taylor's thesis \cite[Remark p.124]{MR2636500}:
\begin{theorem}\label{thm:examples}
Let $F=\IQ(\sqrt{-2})$ and $p=3$. The Hida families passing through the $p$-ordinary cuspidal Bianchi modular forms of levels $\Gamma_0(7+\sqrt{-2})$, $\Gamma_0(9+3\sqrt{-2})$ and $\Gamma_0(4+7\sqrt{-2})$ contain only finitely many classical automorphic forms. 
\end{theorem}
This result also addresses the questions raised in \cite[Section 4]{Loeffler:2010ab}. The computations are not limited to the examples provided and one could use our results and code to produce further examples along the lines of Theorem \ref{thm:examples} for other imaginary quadratic fields and primes $p$, as well as higher tame level. We note however that the necessary computations of spaces of Bianchi modular forms with nebentypus and the computation of Hecke eigenvalues quickly become computationally expensive as the level increases. \\
Finally, we remark that it would be interesting to examine if the third possibility in Theorem \ref{thm:support} can be removed, which would make it much easier to prove that any given family contains finitely many classical automorphic forms. It would also be interesting to examine what can be established along the lines of Theorem \ref{thm:support} for arbitrary number fields, though the analogous statement on the codimension of the support of $\IT$ is only conjectural in general. 

 \section{Setup}\label{sec:one}
 Let $F$ be a number field which, unless otherwise specified, is imaginary quadratic, and let $\calO_F$ denote its ring of integers. Our main focus will be on automorphic forms on $\GL(2)/F$ contributing to the cohomology of some congruence subgroup $\Gamma\subset \SL_2(\calO_F)$ and we assume $\Gamma$ is torsion--free. The relevant locally symmetric spaces in this case are quotients $\Gamma\backslash\mathbb{H}_3$ of the hyperbolic $3$-space, which can be viewed as a subset of the Hamilton quaternions $\mathbb{H}$. Writing $z\in\mathbb{H}$ as $z=a+bi+cj+dk$ for $a,b,c,d\in\IR$, we have
$$\mathbb{H}_3=\{z=a+bi+cj\vert a,b\in\IR,  c\in\IR^{>0}\}\subset \mathbb{H}.$$ 
The action of $\Gamma$ can then be written succintly with multiplication and inversion in $\mathbb{H}$ as
$$\begin{pmatrix}
a&b\\
c&d\\
\end{pmatrix}
\cdot z =(az+b)(cz+d)^{-1}.
$$

 Let $p$ denote a rational prime, we fix algebraic closures $\overline{\IQ}$ and $\overline{\IQ}_p$ together with an embedding $i:\overline{\IQ}\hookrightarrow\overline{\IQ}_p$ and denote by $E$ a finite extension of $\IQ_p$ with $p$-adic ring of integers $\calO$. For an $\calO_F$-module $A$, we consider for positive integers $k,k'$ the irreducible representations
$$V_{k,k'}(A)=\Sym_k(A^2)\otimes \overline{\Sym_{k'}}(A^2)$$
of $\SL_2(A)$ obtained by taking symmetric powers of the standard representation, with the action on the second factor twisted by complex conjugation. When $A=\IC$ or $E$, these exhaust the isomorphism classes of finite dimensional irreducible representations, as a real Lie group or as $E$-rational representations, respectively. We write $Y(\Gamma)$ for the Borel-Serre compactification of the associated locally symmetric space and $\underline{M}$ for the local system on $Y(\Gamma)$ associated to a $\Gamma$-module $M$. We also denote by $H^i_{c} (Y(\Gamma), \underline{M})$ the cohomology with compact support, and consider the resulting boundary long exact sequence:
$$\xymatrix{\ar[r]&H^{i-1}(\partial Y(\Gamma), \underline{M})\ar[r]&H^i_{c} (Y(\Gamma), \underline{M})\ar[r]&H^i (Y(\Gamma), \underline{M})\ar[r]&H^i (\partial Y(\Gamma), \underline{M})\ar[r]&}$$
Following Hida, we denote by $H^1_{P} (Y(\Gamma), \underline{M})$ the image of $H^1_{c} (Y(\Gamma), \underline{M})$ in $H^1(Y(\Gamma), \underline{M})$. Under our assumptions, $Y(\Gamma)=\overline{\Gamma\backslash\mathbb{H}_3}$ is a classifying space for $\Gamma$ and we have that
  $$H^1(Y(\Gamma), \underline{V_{k,k'}(\IC)})\cong H^1(\Gamma,V_{k,k'}(\IC)).$$
 By the Eichler-Shimura isomorphism, the cohomology $H^1_{P}(Y(\Gamma), \underline{V_{k,k'}(\IC)})$ can be computed via the $(\ig,K)$-cohomology of cuspidal automorphic representations $\Pi=\Pi_\infty\otimes\Pi_f$. Moreover, the $(\ig,K)$-cohomology vanishes unless the archimedean component $\Pi_\infty$ contributes to cohomology with coefficients in a unitary representation of $\SL_2(\IC)$ at each archimedean place. This forces the coefficients to be conjugate self-dual, so that $H^1_{P}(Y(\Gamma), \underline{V_{k,k'}(\IC)})=0$ vanishes unless $k=k'$. The space of weight $k$ cusp forms, which we will be mostly concerned with, is then isomorphic to $H^1_{P}(Y(\Gamma), \underline{V_{k,k}(\IC)})$. See \cite[Section 3.6]{MR892187} for details.\\
 Via the embedding $i$, we may also consider cohomology with coefficients in the $p$-adic representations  $V_{k,k'}(E/O)$, which serve as an avatar for spaces of $p$-adic automorphic forms. The standard Hecke operators act on these spaces, and we use the subscript $H^i_{\ord}$ for the \emph{ordinary} part of the cohomology, namely the direct summand of the $p$-adic cohomology that $U_p$ acts faithfully on. 
 \begin{remark}\label{rem:simplify}
 \begin{enumerate}
 \item
 To simplify exposition and exhibit the analogy with modular forms, we will limit some statements to automorphic forms for congruence subgroups of $\SL_2(\calO_F)$ and sometimes assume $F$ has class number one. Everything goes through without the latter assumption, whereas the whole theory of ordinary families is carried out for quaternion algebras $B$ over $F$ in \cite{MR1203228} and we will reference some pertinent results when $B\neq M_2(F)$, with analogous notations. For $\GL_2$, the more natural condition is to consider nearly ordinary cohomology and Hecke algebras, and we refer the reader to \cite{MR1313784} for the details on lifting the results from $\SL_2$ to $\GL_2$. 
 \item
 We mention the more general adelic setting if $F$ has class group $\calC_F$ of order $h_F$. Letting $\mathbb{A}_f$ denote the ring of finite adeles over $F$ and $U$ an open subgroup of $\GL_2(\widehat{\calO}_F)$ small enough so that the arithmetic groups are torsion--free, one can show that the locally symmetric space 
 $$S(U)=\GL_2(F)\backslash (\GL_2(\mathbb{A}_f)/U)\times \mathbb{H}_3$$ 
 is a disjoint union of arithmetic hyperbolic threefolds. Assuming for simplicity the determinant map $\det:U\to \widehat{\calO}_F^\times$ is surjective, we may write $S(U)=\sqcup_{i=1}^{h_F} \Gamma_i\backslash \mathbb{H}_3$, where $\Gamma_i=U\cap M_0(\ib_i)^\times$ for ideals $\ib_i$ forming a set of representatives of $\calC_F$ and maximal orders $M_0(\ib_i)=\{
\begin{pmatrix}
a&b\\
c&d\\
\end{pmatrix}
\in M_2(F) \vert a,d \in \calO_F, b\in \ib_i^{-1}, c\in\ib_i\}
$ of $M_2(F)$. The cohomology of $S(U)$ then decomposes into a direct sum $\bigoplus_{i=1}^{h_F}H^q_{*}(\Gamma_i\backslash \mathbb{H}_3, \underline{M})$ and Hecke operators associated to non-principal ideals act on the cohomology, permuting the terms of the sum. Nevertheless, systems of Hecke eigenvalues associated to non-principal ideals differ from principal ones only by a character of the class group $\calC_F$. 
For details on the general setting and the definition of automorphic forms over imaginary quadratic fields we refer the reader to J. Bygott's thesis  \cite[Sections 5-6]{bygott1998modular}. 
 
 \end{enumerate}
 \end{remark}

Hida constructs nearly ordinary Hecke algebras for $\GL(2)/F$ and arbitrary number fields $F$ and proves the full control theorem  \cite[Theorem 3.2]{MR1313784}. For $\SL(2)$ over an imaginary quadratic field, this already appears in Taylor's work \cite[Section 4]{MR2636500}.  We recall some of the main features, borrowing from Hida's notations:
let $B$ denote a quaternion $F$-algebra split at $p$ and $\infty$, let $\calO_p$ denote $\calO_F\otimes\IZ_p$ and $R$ denote a fixed maximal order of $B$ together with an isomorphism $i_p:R\otimes\IZ_p\to M_2(\calO_p)$. Finally, let $\nu:B\to F$ denote the reduced norm. For an integral ideal $N$ prime to $p$, fix some torsion--free subgroup $\Gamma$ of the norm one elements in $R^\times$ containing $R^\times(N)=\{\gamma\in R^\times\vert \gamma-1\in NR, \nu(\gamma)=1\}$ and define 
\begin{align*}\Delta_0(p^\alpha):=\{\gamma \in \Gamma\vert i_p(\gamma)&= \begin{pmatrix}a&b\\c&d\end{pmatrix} \text{ with }c\in p^\alpha \calO_p\}\\
 \Delta_1(p^\alpha):=\{\gamma \in \Gamma\vert i_p(\gamma)&= \begin{pmatrix}a&b\\c&d\end{pmatrix} \text{ with }c,d-1\in p^\alpha \calO_p\}.
 \end{align*}
When $B=M_2(F)$, one can just take the usual congruence subgroups $\Gamma_0(N)\cap\Gamma_1(p^\alpha)<\SL_2(\calO_F)$ for suitable $N$. \par

For a character $\varepsilon$ of $\Delta_0(p^\alpha)/\Delta_1(p^\alpha)\cong (\calO_p/p^{\alpha}\calO_p)^\times$ we denote by $V_{k,k',\varepsilon}$ the representations where the action of $\gamma\in\Delta_0(p^\alpha)$ on $V_{k,k'}$ is twisted by $\varepsilon(d)$ for $d$ the lower diagonal entry in $i_p(\gamma)\in M_2(\calO_p)$. Hida defines $p$-ordinary cohomology groups $H^i_{\ord}(\Delta_1(p^\infty),V_{k,k',\varepsilon}(E/\calO))=\varprojlim_\alpha H^i_{\ord}(\Delta_1(p^\alpha),V_{k,k',\varepsilon}(E/\calO))$, which are naturally modules over the completed group ring $\calO[[\calO_p^\times]]=\varinjlim_\alpha \calO[(\calO_p/p^\alpha\calO_p)^\times]$. Each continuous character of $\calO_p^\times$ (valued in $\overline{\IZ}_p$) extends to an algebra homomorphism of $\calO[[\calO_p^\times]]$ whose kernel defines a point on $\Spec(\calO[[\calO_p^\times]])$, the space of $p$-adic weights. Up to components, this is just the spectrum of the completed group ring $\Lambda=\calO[[\Gamma_F]]$ of the torsion-free part $\Gamma_F$ of $\calO_p^\times$. Complex conjugation acts on $\Gamma_F$ which decomposes into eigenspaces $\Gamma_F^{+}\times\Gamma_F^{-}$ each non-canonically isomorphic to $\IZ_p$. After choosing topological generators, we may identify $\Lambda\cong \calO[[\IZ_p^2]]\cong \calO[[X,Y]]$. We are interested in arithmetic characters, which are of the form
\begin{align*}\chi: \calO_p^\times&\to \overline{\IZ}_p\\
x&\mapsto x^k\overline{x}^{k'}\varepsilon(x)\\
\end{align*}
for integers $k,k'\geq 0$ and $\varepsilon$ a finite order character of $p$-power conductor. We denote the corresponding point on $\Lambda$ by $P_{k,k',\varepsilon}$. We remark that since we are interested in finiteness results, we mostly just work with one component $\Lambda$ of weight space. For any $\Lambda$-module $M$ and prime ideal $P$ of $\Lambda$ we write $M[P]$ for the submodule annihilated by $P$. Hida shows the following results for $p$-ordinary cohomology groups:

\begin{theorem}[see Theorems I,II and 5.1 of \cite{MR1203228}] \label{prop:allweightsIQF} \ \\
\begin{enumerate}
\item The modules $H^1_{\ord}(\Delta_1(p^\infty), V_{k,k',\varepsilon}(E/\calO))$ and $H^1_{P,\ord}(\Delta_1(p^\infty), V_{k,k',\varepsilon}(E/\calO))$ are of cofinite type over $\Lambda$ and independent of the weight: 
$$H^1_{*,\ord}(\Delta_1(p^\infty), V_{k,k',\varepsilon}(E/\calO))\cong H^1_{*,\ord}(\Delta_1(p^\infty), E/\calO).$$

\item The natural homomorphisms
$$H^1_{\ord}(\Delta_1(p^\alpha), V_{k,k',\varepsilon}(E/\calO))\to H^1_{\ord}(\Delta_1(p^\infty), E/\calO)[P_{k,k',\varepsilon}]$$
$$H^1_{P,\ord}(\Delta_1(p^\alpha), V_{k,k',\varepsilon}(E/\calO))\to H^1_{P,\ord}(\Delta_1(p^\infty), E/\calO)[P_{k,k',\varepsilon}]$$
are Hecke-equivariant, the first being an isomorphism and the second having finite kernel and cokernel.
\end{enumerate}
\end{theorem}
Letting $\IT^{\ord}_{k,k'}(p^\alpha;\calO;\varepsilon)$ denote the $\calO$-subalgebra of $\End_\calO(H^1_{P,\ord}(\Delta_1(p^\alpha), V_{k,k',\varepsilon}(E/\calO))$ generated by the Hecke operators, one therefore obtains a universal ordinary Hecke algebra by setting $\IT:=\varinjlim_\alpha\IT^{\ord}_{k,k'}(p^\alpha;\calO)$, together with a control theorem stating that the natural map 
\begin{equation}\IT/P_{k,k',\varepsilon}\IT\to \IT^{\ord}_{k,k'}(p^\alpha;\calO;\varepsilon)\label{map:sp.}\tag{sp.}
\end{equation}
where $\varepsilon$ factors through $(\calO_p/p^\alpha\calO_p)^\times$, has finite kernel and cokernel. We refer the reader to \cite[Theorem 3.2]{MR1313784} for the full control theorem for nearly ordinary Hecke algebras for $GL(2)$.
Since $H^1_{P}(\Delta_i(p^\alpha), V_{k,k'}(\IC))=0$ unless $k=k'$, one obtains:
\begin{cor}\label{cor:parallel} The specialization of $\IT$ at arithmetic points $P_{k,k',\varepsilon}$ is finite as soon as $k\neq k'$. 
\end{cor}
In particular, the Hecke algebra $\IT$ cannot be free over $\Lambda$, but has to be a torsion module \cite[Corollary 4.4]{MR2636500}. We recall that in contrast when $F=\IQ$, Hida constructs universal $p$-adic ordinary Hecke algebras which are finitely generated and free over weight space $\calO[[\IZ_p^\times]]$. Together with the control theorems this implies the specialization at \emph{any} points on $\calO[[\IZ_p]]\cong \calO[[T]]$ coming from arithmetic characters $x\mapsto x^k\varepsilon(x)$ on $\IZ_p^\times$ recovers the system of Hecke eigenvalues of a classical eigenform, where $k\geq 0$ and $\varepsilon$ has finite $p$-power order (this corresponds to weight $\geq2$ in the usual terminology for modular forms). For arbitrary number fields, we will call such points on weight space \emph{classical points}. The constructions and results above exist more generally for arbitrary number fields $F$ and cohomology in the appropriate degrees. It follows similarly that, as soon as $F$ has a complex place, the relevant Hecke algebra as a module over $\calO[[(\calO_F\otimes\IZ_p)^\times]]$  is torsion \cite[Theorems 2.5, 5.2]{MR1313784}, whereas for totally real fields the behavior resembles the case of modular forms and the Hecke algebras are torsion-free \cite[Theorem II]{MR960949}. More precisely, Hida conjectures:

\begin{conj} [\cite{MR1313784}, Conjecture 4.3]\label{conj:codim}
The image of the nearly ordinary Hecke algebra in weight space has codimension the number of complex places of $F$.
\end{conj}
For CM components, we note that this amounts to Leopoldt's Conjecture for the respective number field. As soon as $F$ has a complex place, it is therefore no longer clear whether or not, as for totally real fields, the Hecke algebras interpolate a Zariski-dense set of systems of Hecke eigenvalues corresponding to classical automorphic forms. The results of \cite{MR2461903} on deformations of Galois representations together with $R=\IT$ considerations suggest that this is not to be expected in general. Even for automorphic representations on $\GL(1)$, for general number fields density does not necessarily hold \cite[Section 8.4]{MR2461903}. See also \cite{Loeffler:2010ab}, where it is shown that when $F$ is not totally real and does not contain a CM subfield, algebraic Gr\"ossencharacters of $F$ are not dense on the eigenvariety in the rigid analytic topology. 
For $\GL(2)/F$ when $F$ is imaginary quadratic, it is thus natural to ask: 

\begin{question}\label{question:dense}
When does the support in $\Lambda$ of a component of $\IT$ contain a Zariski-dense set of classical points? 
\end{question}

It follows from Corollary \ref{cor:parallel} that the points on $\Spec(\Lambda)$ that can give rise to classical cuspidal automorphic forms have parallel weight. Recall the decomposition into eigenspaces for complex conjugation $\Gamma_F^+\times\Gamma_F^-$ of the torsion-free part $\Gamma_F$ of $\calO_p^\times$. We denote by $\varepsilon^\pm$ the restriction of a $p$-power order character $\varepsilon$ of $\calO_p^\times$ to the respective piece. After choosing topological generators and suitably normalizing the $\Lambda=\calO[[\Gamma_F]]\cong\calO[[X,Y]]$-module structure, we may then assume these points are of the form 
\begin{align*}P_{k,\zeta,\zeta'}:&=(X+1-\zeta(1+p)^k, Y+1-\zeta'(1+p)^k)\Lambda\\
&=(X+1-\varepsilon^{+}(1+p)(1+p)^k, Y+1-\varepsilon^{-}(1+p)(1+p)^k)\Lambda
\end{align*}
 for integers $k\geq 0$ and $p$-power roots of unity $\zeta, \zeta'\in \mu_{p^\infty}(\overline{\IQ}_p)$.

Moreover, it is known that for imaginary quadratic fields, Conjecture \ref{conj:codim} holds.  Namely, when $B=M_2(F)$ and $F$ is imaginary quadratic, Hida shows:
 \begin{theorem}[Theorem 6.2 of \cite{MR1313784}]\label{thm:codimone}
 The Pontryagin dual $M$ of the ordinary parabolic cohomology group $H^1_{P,\ord}(Y(\Delta_0(p^\alpha)),\underline{\calC})$ is a $\Lambda$-torsion module of homological dimension one. 
 \end{theorem}
 Here, $\calC$ denotes the space of continuous functions on the column space $X=(\calO_p^\times, p\calO_p)^t$ with values in $E/\calO$. Effectively, the cohomology with coefficients in the locally constant sheaf $\underline{\calC}$ associated to $\calC$ computes the relevant cohomology at infinite level:
 \begin{align*} H^1_{*, \ord}(Y(\Delta_0(p)),\underline{\calC})&\cong H^1_{*, \ord}(Y(\Delta_1(p^\infty)),\underline{E/\calO})\\
 &= \varinjlim_{\alpha} H^1_{*, \ord}(Y(\Delta_1(p^\alpha)),\underline{E/\calO}).\\
 \end{align*}
 In particular, following Hida one may take the universal ordinary Hecke algebra $\IT$ to be the $\calO$-subalgebra of $\End_{\calO}(H^1_{P,\ord}(Y(\Delta_0(p)),\underline{\calC})))$ generated by the Hecke operators. 
 \begin{remark}
 When $B\neq M_2(F)$ is an indefinite division algebra, Hida also proves that the dual $M$ of $H^1_{\ord}(Y(\Delta_0(p^\alpha)),\underline{\calC})$, where $Y(\Delta_0(p^\alpha))$ is now compact, is killed by a non-zero element of $\Lambda$ and has homological dimension one. The same holds true after lifting to $\GL_2$ and considering nearly ordinary cohomology, provided $p>2$ (see \cite[Theorem 5.2]{MR1313784} and the ensuing discussion).
 \end{remark}
 In both cases the respective (nearly) ordinary Hecke algebra $\IT$ acts faithfully on the (nearly) ordinary cohomology and embeds into the respective dual $M$. Since $M$ is finitely generated over $\IT$, letting $m$ denote the minimum number of generators we have the divisibility of characteristic power series: $\ch(\IT)\vert \ch(M)\vert \ch(\IT)^m$. Since we will be analyzing the zeroes of these power series to study the support of our $p$-adic families in weight space, we may simply work with $\ch(\IT)$, and we record here, for some implicit tame level $N$, for $B=M_2(F)$ or $B$ a division algebra, and $F$ an imaginary quadratic field:
 \begin{cor}\label{cor:charpow}
 The (nearly) ordinary Hecke algebra $\IT$ is annihilated by a non-zero power series $\phi\in\calO[[X,Y]]$. 
 \end{cor}
 Thus, to answer question \ref{question:dense} one is led to study which cyclic modules $\Lambda/\phi$ contain a dense set of arithmetic points that can give rise to classical cuspidal automorphic forms. The Zariski closure of the points $P_{k,\zeta,\zeta'}$ on a component of the support of $\IT$ either contains the diagonal in $\Spec(\Lambda)$ or intersects it in finitely many points. However, one cannot a priori exclude the case that the component contains infinitely many points of varying non-parallel Nebentypus, and we study this in more detail in Section \ref{sec:two}. \par
It can indeed be the case that the support of $\IT$ contains the diagonal via functoriality constructions: assuming $p$ is split, the base--change of a Hida family over $\IQ$ gives rise to a component of the Hecke algebra of the same tame level with classical automorphic forms in each parallel weight $k\geq 0$. Similarly, such examples may arise via theta lifts from suitable Gr\"ossencharacters of quadratic extensions $L/F$. We refer the reader to  \cite[Proposition 4.4]{MR2636500} for details. 
On the Galois side, which is equally mysterious, Calegari and Mazur conjecture in \cite[Conjecture 1.3]{MR2461903}:
\begin{conj} Let $F$ an imaginary quadratic field and $p$ a split prime. Suppose a representation $\rho:\Gal(\overline{F}/F)\to\GL_2(E)$ is continuous, irreducible, nearly ordinary and unramified outside a finite set of places and admits infinitesimally classical deformations. Then one of the following holds: 
\begin{itemize}
\item[``BC''] There exists a character $\chi$ such that $\chi\otimes\rho$ descends to a representation over $\IQ$. \\
\item[``CM''] The projective image of $\rho$ is dihedral, and the determinant character descends to $\IQ$ with a CM fixed field in which a prime above $p$ splits. \\
\end{itemize}
\end{conj}
An infinitesimally classical deformation in the language of \cite{MR2461903} is a point of the infinitesimal nearly ordinary deformation ring $R$ of $\rho$ that gives rise to a $\IQ$-rational line on the tangent space (at weight zero) of the $E$-vector space of infinitesimal $p$-adic weights. In particular, if a one-parameter family of nearly ordinary deformations arose from classical automorphic forms over $F$, they would have rational infinitesimal Hodge-Tate weights and give rise to infinitesimally classical deformations. \\
Calegari and Mazur prove the conjecture assuming a stengthening of Leopoldt's Conjecture when $\rho$ has finite image \cite[Theorem 1.2]{MR2461903}, as well as the existence under suitable hypotheses of one-parameter families of nearly ordinary deformations with finitely many crystalline specializations of parallel weight \cite[Theorem 1.5]{MR2461903}. Therefore, guided by $R=\IT$ conjectures, we are inclined to believe the automorphic analogues of these results and seek to prove the existence of families with finitely many classical points. On the other hand, significant recent progress in modularity \cite{2018arXiv181209999A} suggests that it may be possible to show that the automorphic and Galois-theoretic questions are equivalent.

 \section{Algebraic criterion}\label{sec:two}
 Let $\calO$ denote the valuation ring of a finite extension of $\IQ_p$. In this section we establish the key rigidity properties for two-variable formal power series over $\calO$ vanishing at points $P_{k,\zeta,\zeta'}:=((1+p)^k\zeta-1,(1+p)^k\zeta'-1)\in \overline{\IQ}_p^2$ for $p$-power roots of unity $\zeta,\zeta'$ and positive integers $k$, building upon previous work of the author. These statements are inspired by rigidity results appearing in Hida's work (see e.g., \cite[Section 4]{MR3220926}). We show the following:
 \begin{prop}\label{prop:twovark}
Let $\phi\in \calO[[X,Y]]$ be a power series. Suppose that
$$\phi((1+p)^{k}\zeta-1,(1+p)^{k}\zeta'-1)=0$$
for an infinite set $\Sigma$ of triples $(k,\zeta,\zeta')$  in $\IZ_p\times \mu_{p^\infty}^2$. We then have the following:
\begin{enumerate}
    \item We may find fixed $N,K\in \IZ_p$  and $\xi\in\mu_{p^\infty}$ such that, up to possibly swapping coordinates, we have for every $\zeta\in \mu_{p^\infty}$ the vanishing:
$$\phi((1+p)^K\zeta-1,(1+p)^K\xi\zeta^N-1)=0.$$
\item If, in addition to that, $\phi$ is irreducible and passes through the origin, then $\phi(X,Y)=(X+1)^N-(Y+1)$ for some $N\in\IZ_p$ (up to coordinate swap); in particular, in this case there are infinitely many weights $k$ in $\Sigma$ if and only if $\phi$ vanishes along the diagonal.
\end{enumerate}

 \end{prop}
This, together with our previous considerations, yields the main result of this section: 
\begin{theorem}\label{thm:mainalgebraic}
Let $F$ be an imaginary quadratic field of class number one and $p$ a prime split in the ring of integers $\calO_F$. Let $\varphi$ be a cuspidal automorphic form on $\GL(2)$ over $F$ (nearly) ordinary at $p$ of cohomological type. Let $\calH$ denote the cuspidal (nearly) ordinary $p$-adic family containing $\varphi$. Normalize the $\Lambda$-module structure so that $\varphi$ corresponds to $X=Y=0$ on $\Lambda\cong \calO[[X,Y]]$. Then the component $D$ of the support of $\calH$ in $\Lambda$ passing through $\varphi$ either:
\begin{enumerate}
\item contains the diagonal $X=Y$ or
\item contains only finitely many points that give rise to classical automorphic forms or 
\item contains only finitely many points $P_{k,\zeta,\zeta'}$ with weight $k$ different from the weight of $\varphi$ but an infinite collection of points with the same weight as $\varphi$ ($k_\varphi=0$ in $\Lambda$ with this normalization) and varying non-parallel $p$-power nebentypus characters. Moreover, there is some $N\in \IZ_p$ such that $D$ contains $P_{0,\zeta,\zeta^N}$ (or $P_{0,\zeta^N,\zeta}$) for every $\zeta\in \mu_{p^\infty}$. 
\end{enumerate}
\end{theorem}
\begin{proof}
The support of $\calH$ has pure codimension one in $\Lambda$ and by Corollary \ref{cor:charpow} the family has a characteristic power series.  By our assumptions on $D$ it suffices to consider its irreducible factor $\phi$ passing through the origin. By Proposition \ref{prop:twovark}, if $\phi$ vanishes along infinitely many points $P_{k,\zeta,\zeta'}$ that give rise to classical cuspidal automorphic forms, then $\phi$ either vanishes along the diagonal or is a unit multiple of $(X+1)^N-(Y+1)$ for some $N\in \IZ_p$ and up to swapping $X$ and $Y$. The result follows. 
\end{proof}
\begin{remark}\label{rem:generalh}
The assumption on the class number is not necessary in Theorem \ref{thm:mainalgebraic}. However we have not given a detailed proof in the general case to simplify exposition. It is easy to see that, since the cohomology of the relevant locally symmetric spaces as in Remark \ref{rem:simplify} decomposes into $h_F$ pieces, modulo a suitable character of the class group studying the Hecke eigensystems amounts to studying them on each component. Since Hida's codimension one results are established also for $B\neq M_2(F)$ and Corollary \ref{cor:charpow} remains valid one may again apply the rigidity results of Proposition \ref{prop:twovark} to each of the finitely many components of the locally symmetric space and establish a more general result along the lines of Theorem \ref{thm:mainalgebraic}. 

\end{remark}
We deal with the set of special points of Proposition \ref{prop:twovark} as follows: assume the weights $k$ in $\Sigma$ are becoming arbitrarily small $p$-adically. Then the points $((1+p)^{k}\zeta-1,(1+p)^{k}\zeta'-1))$ are converging $p$-adically toward $(\zeta-1,\zeta'-1)$, which are precisely the coordinates of torsion points on the formal multiplicative group $\widehat{\IG}_m^2$. For these formal multiplicative groups, we have shown a $p$-adic formal analogue of the multiplicative Manin--Mumford statement (see e.g., \cite{MR0190146}) holds, implying that subschemes approaching infinitely many torsion points must contain a translate of a formal subtorus. More precisely, we use the following special case of \cite[Theorem 1.3]{Serban:2016aa}:
\begin{prop}\label{prop:mmm}
Write $A=\calO[[X,Y]]/(\phi)$ for some nonzero $\phi\in \calO[[X,Y]]$. There exists in $\Spf(A)$ a finite union $\calT$ of translates of proper formal subtori by torsion points of $\widehat{\IG}_m^2$ and a constant $C_\phi>0$ depending only on $\phi$ and the choice of $p$-adic absolute value such that
$$\vert \phi(\zeta-1,\zeta'-1)\vert_p>C_\phi$$
provided $(\zeta-1,\zeta'-1)\in \widehat{\IG}_m^2[p^\infty]\setminus(\calT(\overline{\IQ}_p)\cap\widehat{\IG}_m^2[p^\infty])$. 
\end{prop}
We therefore deduce:
 \begin{lemma}\label{lemma:ktozero}
 Let $\phi\in \calO[[X,Y]]$ be a power series. Suppose we have the vanishing 
$$\phi((1+p)^{k_j}\zeta_j-1,(1+p)^{k_j}\zeta_j'-1)=0$$
for a sequence of $p$-adic integers $\{k_j\}_{j\in\IN}$ converging to $0$ in the $p$-adic metric and for a sequence of pairs of $p$-power roots of unity $\{(\zeta_j,\zeta_j')\}_{j\in\IN}$ such that the set $\{(k_j,\zeta_j,\zeta_j')\vert j\in\IN\}\subset \IZ_p\times \mu_{p^\infty}^2$ is infinite. Then $\phi$ is divisible by $\xi(X+1)^N-(Y+1)$ over $\calO[\xi]$ for some fixed $\xi\in\mu_{p^\infty}$ and $N\in\IZ_p$, up to possibly swapping coordinates. Moreover, if $k_j\neq 0$ infinitely often and $\phi$ is irreducible passing through the origin, then one must have $N=\xi=1$ and $\phi$ vanishes along the diagonal.
 \end{lemma}
 \begin{proof} 
 First observe that we may assume that the set $\{(\zeta_j,\zeta_j')\vert j\in\IN\}$ is infinite. Indeed, otherwise we obtain the vanishing $\phi((1+p)^{k_j}\zeta-1,(1+p)^{k_j}\zeta\xi-1)=0$ for infinitely many $k_j$ and fixed $\zeta,\xi\in\mu_{p^\infty}$ (up to coordinate swap). But this infinite set of zeroes is by Weierstrass preparation Zariski-dense in $\Spec(\calO[\xi][[X,Y]])/(\xi(X+1)-(Y+1)))$, and the result follows. \\
 Assume now for nonzero $\phi\in \calO[[X,Y]]$ that the set $\{(\zeta_j,\zeta_j')\vert j\in\IN\}$ is infinite and that $C_\phi$ is as in Proposition \ref{prop:mmm}. Since $\vert (1+p)^{k_j}\zeta_j-\zeta_j\vert_p\leq p^{-v_p(k_j)}$, we may by continuity of $\phi$ assume $j\in \IN$ is large enough so that $\vert \phi(\zeta_j-1,\zeta_j'-1)\vert_p\leq C_\phi$. By Proposition \ref{prop:mmm} it must therefore be that $\Spf (\calO[[X,Y]]/\phi)$ contains a translate by a torsion point of a (strictly) positive-dimensional formal subtorus. In particular, Proposition \ref{prop:mmm} implies there are fixed $N\in\IZ_p$ and $(\xi_1-1, \xi_2-1)\in  \widehat{\IG}_m^2[p^\infty]$ such that  for all $\zeta\in \mu_{p^\infty}$ we have that
 $$\phi(\xi_1\zeta-1, \xi_2\zeta^N-1)=0,$$ 
after possibly switching the roles of $X$ and $Y$. We then set $\xi:=\xi_1^{-N}\xi_2$ and may write 
 $$\phi(X,Y)=H(X)+(\xi(X+1)^N-(Y+1))G(X,Y)$$
 for $H(X)=\phi(X,\xi(X+1)^N-1)\in \calO[\xi][[X]]$ and $G(X,Y)\in \calO[\xi][[X,Y]]$. Since $H(\xi_1\zeta-1)=0$ for all $\zeta\in \mu_{p^\infty}$, we must have that $H(X)=0$ by Weierstrass preparation. This proves the first part of the claim. \\
 For the second part, we take the product over all $G_\xi:=\Gal(\calO[\frac{1}{p}][\xi]/\calO[\frac{1}{p}])$-conjugates of the equality $\phi(X,Y)=(\xi(X+1)^N-(Y+1))G(X,Y)$. We may in this equality factor $G(X,Y)=P(X,Y)\cdot Q(X,Y)$ over $\calO[\xi]$ with $P,Q$ chosen so that the prime factors of $P(X,Y)$ consist exactly of all the conjugates $(\sigma(\xi)(X+1)^N-(Y+1))$ for $\sigma\in G_\xi$ that divide $G(X,Y)$. We note that we often use here that our power series rings are unique factorization domains. Taking now the product of $G_\xi$-conjugates, by irreducibility $\phi$ must either divide $\prod_{\sigma\in G_\xi}\sigma\left((\xi(X+1)^N-(Y+1))\cdot P(X,Y)\right)\in\calO[[X,Y]]$ or $\prod_{\sigma\in G_\xi}\sigma\left(Q(X,Y)\right)\in\calO[[X,Y]]$. But the latter is impossible by our choice of $Q(X,Y)$ as none of the Galois conjugates can account for the factor $\xi(X+1)^N-(Y+1)$ of $\phi$ over $\calO[\xi]$. The product $\prod_{\sigma\in G_\xi}\sigma\left((\xi(X+1)^N-(Y+1))\cdot P(X,Y)\right)$ being by definition of $P$ just a number of copies of $\prod_{\sigma\in G_\xi}(\sigma(\xi)(X+1)^N-(Y+1))$ up to units, we deduce that $\phi$ divides $\prod_{\sigma\in G_\xi}(\sigma(\xi)(X+1)^N-(Y+1))\in\calO[[X,Y]]$. 
 Finally, provided $k_j\neq 0$, a factor of $\phi$ of the form $\xi(X+1)^N-(Y+1)$ cannot vanish at $((1+p)^{k_j}\zeta_j-1,(1+p)^{k_j}\zeta_j'-1)$ unless $N=1$, when that factor does not pass through the origin unless $\xi=1$. The result follows.
 \end{proof}
 
 \begin{proof}[Proof of Proposition \ref{prop:twovark}]
We consider for $K\in \IZ_p$ the translate of $\phi$ given by the power series
$$\Phi_K(X,Y):=\phi((1+p)^K(X+1)-1, (1+p)^K(Y+1)-1).$$
If $\Sigma$ contains infinitely many weights $k$, by compactness there is an infinite subsequence $\{k_j\}_{j\in\IN}$ converging $p$-adically to some $K\in \IZ_p$. If finitely many weights occur in $\Sigma$, there exists similarly $K\in\IZ_p$ such that $\Phi_K(\zeta-1,\zeta'-1)=0$ for infinitely many pairs $(\zeta,\zeta')\in \mu_{p^\infty}^2$.\\
 In both cases Lemma \ref{lemma:ktozero} implies that $\Phi_K$ must be, after possibly switching the roles of $X$ and $Y$, divisible over $\calO[\xi]$ by $\xi(X+1)^N-(Y+1)$ for some $\xi\in\mu_{p^\infty}$ and $N\in \IZ_p$.  In particular, for all $\zeta \in  \mu_{p^\infty}$ we have the claimed vanishing 
$$\phi((1+p)^K\zeta-1,(1+p)^K\xi\zeta^N-1)=0.$$

To prove the second part, observe that up to units and swapping $X$ and $Y$, there is therefore an irreducible component of the original power series $\phi$ having infinitely many zeroes in common with the power series
$$\kappa_{N,K,\xi}(X,Y):=(1+p)^{-KN+K}\xi (X+1)^N-(Y+1).$$
This infinite set of zeroes is Zariski-dense in $\Spec(\calO[\xi][[X,Y]]/(\kappa_{N,K,\xi}))$ since the underlying ring $\calO[\xi][[X,Y]]/(\kappa_{N,K,\xi})$ is isomorphic to a one-variable power series ring over $\calO[\xi]$ and thus an infinite set of points is dense by Weierstrass preparation. It follows that $\phi$ must be divisible by $\kappa_{N,K,\xi}$ over $\calO[\xi]$. Taking again $G_\xi:=\Gal(\calO[\frac{1}{p}][\xi]/\calO[\frac{1}{p}])$-conjugates we see that the irreducible power series $\phi$ is divisible by $\prod_{\sigma\in G_\xi}(\kappa_{N,K,\sigma(\xi)})\in\calO[[X,Y]]$.
This establishes the second part of the proposition, since any $\kappa_{N,K,\sigma(\xi)}$ only passes through the origin if $\sigma(\xi)=1$ and either $N=1$ or $K=0$, implying that up to a unit and coordinate swap $\phi$ is of the form $(X+1)^N-(Y+1)$ for some $N\in \IZ_p$.
\end{proof}

%%%%%%%%%%%%%%%%%%%%%%%%%%%%%%%%%%%%
\section{Explicit computations}\label{sec:computations}

We now analyze some explicit examples of Hida families with the aim of proving, when this is true, that the family interpolates only finitely many classical automorphic forms. Our approach is based on Theorem \ref{thm:mainalgebraic}, which allows one to verify via a finite amount of computation that a component of the ordinary Hecke algebra $\IT$ has finitely many classical points. More precisely, Theorem \ref{thm:mainalgebraic} yields the following corollary:
\begin{cor}\label{cor:checktwists}
Let $\IT$ denote the ordinary $p$-adic Hecke algebra of tame level $N$
 for some split prime $p$. Assume a component of $\IT$ passing through an ordinary cuspidal newform $\varphi$ of trivial weight and level dividing $\Gamma_0(Np)$ contains infinitely many classical points. Then either its support in $\Lambda$ contains the diagonal, or for every $i\geq 1$ the Hida family interpolates an additional ordinary cusp form, new at level dividing $\Gamma_0(N)\cap\Gamma_1(p^{i+1})$. 
 \end{cor}
\begin{proof}
Everything follows from Theorem \ref{thm:mainalgebraic}, noting that the characteristic power series $f$ of $\IT$ in $\Lambda$, after normalizing so that $\varphi$ is above the origin in $\Lambda$ and $f(0,0)=0$, has to vanish along a subtorus which we assume is not the diagonal. Up to swapping coordinates we obtain that $f(\zeta-1,\zeta^M-1)=0$ for all $\zeta\in\mu_{p^\infty}(\overline{\IQ}_p)$ and for some fixed $M\in \IZ_p$. This gives rise to primes $P_{0,\zeta,\zeta^M}:=(X+1-\zeta, Y+1-\zeta^M)\Lambda$ for all $\zeta\in\mu_{p^\infty}$ on the support of the Hida family and it is easily seen that regardless of the $p$-adic valuation of $M$, the order of the pairs in $\{(\zeta, \zeta^M)\in\mu_{p^\infty}^2\}$ runs through all of $p^i$ for $i \geq 1$ . Thus we obtain for every $i\geq 1$ a finite order character $\varepsilon$ of $\calO_p^\times$ of conductor $p^i$ such that 
$$P_{0,\zeta,\zeta^M}=(X+1-\varepsilon^{+}(1+p), Y+1-\varepsilon^{-}(1+p))\Lambda$$
is on the support of the Hida family in $\Lambda$. Now apply the specialization map \eqref{map:sp.}, taking into account our specific normalization. It follows from Hida's control theorem that each of the primes $P_{0,\zeta,\zeta^M}$ on the support of the Hida family interpolates a characteristic zero Hecke eigensystem for the algebra generated by the Hecke operators on $H^1_{P,\ord}(\Delta_1(p^{i+1}),V_{0,0,\varepsilon}(E/\calO))$, with notations as in Section \ref{sec:one}.  In our context we simply take $\Delta_1(p^{i+1})=\Gamma_0(N)\cap \Gamma_1(p^{i+1})$. These eigensystems correspond to classical ordinary cuspidal newforms for each $i\geq 1$, as claimed. \end{proof}
It follows that for our purposes we are led to computing with spaces of cuspidal Bianchi modular forms with nebentypus. Computations of spaces of Bianchi modular forms have a rich history, and were necessary to exploring modularity questions (see e.g., \cite{MR743014}), especially in the absence of many algebro-geometric constructions available for modular forms over $\IQ$. In general, there seems to be a paucity of genuine (not twists of base-change or CM automorphic forms) cuspidal Bianchi modular forms. We refer the reader to \cite{MR3091734,Rahm:2017aa} for some large scale computations along those lines. While this expectation is convenient for our intentions, we are not aware of any systematic computations which also include nebentypi. In practice, given an ordinary cuspidal newform in trivial weight $\varphi$ which lifts to an ordinary family of tame level $N$, Corollary \ref{cor:checktwists} leads to the following strategy of checking there are finitely many classical automorphic forms on the family:
\begin{strategy}
\begin{enumerate}
    \item Compute spaces of non-trivial weight $(k,k)$ and level $\Gamma_0(N)$ to find $k>0$ such that there is no cuspidal cohomology in $H^1(\Gamma_0(N),V_{k,k}(\IC))$. 
    \item Compute the dimensions of the new subspaces of cusp forms in trivial weight and level dividing $\Gamma_0(N)\cap\Gamma_1(p^{i+1})$ for $i=1$. For the ordinary newforms (see the criterion in Lemma \ref{lemma:ordinary}), check via a Hecke eigenvalue computation whether they satisfy the necessary congruences to lie on the family. If none do, we are done. 
    \item If one finds a cuspidal newform in the family at level dividing $\Gamma_0(N)\cap\Gamma_1(p^{i+1})$ for $i=1$, repeat the previous step for $i=2,3,...$ until there is no new cuspform to be found at that level lying on the family. 
\end{enumerate}
\end{strategy}

We carry through this strategy for concrete examples, starting with some useful preliminary results. 

\begin{lemma}\label{lemma:ordinary}
Let $\Pi=\bigotimes'_\nu \pi_\nu $ be a cuspidal automorphic representation for $\GL(2)/F$ of conductor $N$ such that the archimedean component $\pi_\infty$ contributes to the cohomology with coefficients in the weight $(k,k)$ representation $V_{k,k}$ for some $k\geq 0$. Let $\pi:=\pi_\ip$, for a prime $\ip$ above an odd rational prime $p$, and suppose $\dim \pi^{\Gamma_1(p^n)}=1$ for some $n\geq 1$. If the corresponding cohomology class $[c]$ in $H^1(\Gamma_1(N), V_{k,k}(\IC))$ is ordinary at $\ip$ then one of the following holds:
\begin{enumerate}
\item The conductor of $\pi$ is $p^n$ and the Nebentypus character of $[c]$ is non-trivial of order exactly divisible by $\ip^{n}$. 
\item The weight $k=0$, $n=1$ and $\pi$ is an unramified twist of the Steinberg representation. If $\dim \pi^{\Gamma_0(p)}=1$ then the twist is unramified quadratic.
\end{enumerate}

\end{lemma}
\begin{proof}
This follows from matching up the different classes of irreducible admissible representations $\pi$ of $\GL_2(\IQ_p)$ with their $U_p$-eigenvalue. Recall that for such representations there is a conductor $c(\pi)$ such that 
$$\dim \pi^{U_1(p^m)}=\max\{0, m-c(\pi)+1\}$$
for all $m\geq 0$. It follows that in our situation we must have $c(\pi)=n$. Since $p>2$ and $\pi$ is infinite dimensional,
$\pi$ is either a principal series representation $\pi(\chi_1,\chi_2)$, a twist of the Steinberg representation $S(\chi)$ or a supercuspidal representation. The supercuspidal case can be excluded since the $U_p$-eigenvalue is trivial in that case (see e.g., \cite{Loeffler:2010aa}). Moreover, the conductors are given by $c(\pi(\chi_1,\chi_2))=c(\chi_1)+c(\chi_2)$ and $c(S(\chi))=\max\{1, 2c(\chi)\}$. If the conductor  $c(\pi)>1$, $\pi$ has non-zero $U_p$-eigenvalue if and only if the associated Weil-Deligne representation $WD(\pi)$ has a non-trivial subspace fixed by inertia (see e.g., \cite[Lemma 4.2.2]{MR1639612}). In the principal series case $WD(\pi)=\chi_1\oplus \chi_2$ and this implies exactly one of the characters, say $\chi_1$, is unramified. The $U_p$-eigenvalue on that space is then just $\chi_1(\Frob_p)$. Letting now $\omega_p$ denote the $p$-part of the central character of $\Pi$, in the principal series case we have that $\omega_p=\chi_1\cdot\chi_2$ for some appropriate normalization. Since $\chi_1$ is unramified and $\chi_2$ has conductor $n$, the same holds for $\omega_p$ and the nebentypus has to have order divisible by $\ip^n$. In the special representation case, $WD(\pi)=\chi\oplus \chi\vert-\vert^{-1}$ (and non-trivial monodromy) and thus the $U_p$-eigenvalue is non-trivial if and only if $\chi$ is unramified, which doesn't happen if $c(\pi)>1$. This leaves the case of $c(\pi)=1$. For principal series, again the two conductors are $c(\chi_1)=0$ and $c(\chi_2)=1$ and we conclude as for $c(\pi)>1$. Finally in the unramified twist of the Steinberg representation case, a computation shows that the $U_p$-eigenvalue is a $p$-adic unit (actually $\chi(p)$) if and only if $k=0$; moreover the $p$-part of the central character is given by $\chi^2$, and the result follows. 
\end{proof}

To compute the number of cuspidal newforms, we subtract the dimensions of the space of Eisenstein series and of oldforms from the computed dimensions of the relevant cohomology groups. Denote by $M_{k,k}(\Gamma_1(N),\varepsilon)\cong H^1(\Gamma_1(N), V_{k,k}(\IC))$ and $S_{k,k}(\Gamma_1(N),\varepsilon)$ the finite dimensional $\IC$-vector spaces of parallel weight $k$ forms and cusp forms with character $\varepsilon$. We have the following:
\begin{prop}\label{prop:eisold}
Let $F$ be imaginary quadratic of class number one, $\Gamma_1(N)\leq \SL_2(\calO_F)$ the usual congruence subgroup for an ideal $N\subset \calO_F$, and $\varepsilon:(\calO_F/N)^\times\to \IC^\times$ a Dirichlet character modulo $N$ such that $\varepsilon(-1)=(-1)^k$. Then:
\begin{enumerate}
\item The dimension of the space of Eisenstein series of level $\Gamma_1(N)$, weight $(k,k)$ and Nebentypus $\varepsilon$ equals 
\begin{align*}\nu_\infty(\varepsilon)
&=-\delta_{k,2}\cdot\delta_{\mathfrak{f}(\varepsilon),1}+\sum_{\substack{d\vert N\\ (d,N/d)\vert N/\mathfrak{f}(\varepsilon)}}\vert (\calO_F/(d,N/d))^\times\vert\\
&=-\delta_{k,2}\cdot\delta_{\mathfrak{f}(\varepsilon),1}+\prod_{\substack{\ip\text{ prime}\\{\ip\vert N}}}\lambda(v_\ip(N),v_\ip(\mathfrak{f}(\varepsilon)),\calN(\ip)),
\end{align*}
where $\mathfrak{f}(\varepsilon)$ denotes the conductor, $\calN$ the norm and for $s\geq 0$, $r\geq\max\{s,1\}$ we set:
$$\lambda(r,s,l)=\begin{cases}l^{r'}+l^{r'-1}& 2s\leq r=2r'\\
2l^{r'}& 2s\leq r=2r'+1\\
2l^{r-s}&2s>r.\\
\end{cases}$$
\item 
There is a maximal Hecke-invariant subspace $S_{k,k}(\Gamma_1(N),\varepsilon)_\text{new}$ on which the Hecke algebra acts semisimply. If $M\vert N$, there are degeneracy maps $\alpha_d:S_{k,k}(\Gamma_1(M))\to S_{k,k}(\Gamma_1(N))$ for every divisor $d$ of $N/M$ such that 
$$S_{k,k}(N,\varepsilon)=\bigoplus_{\substack{M\vert N\\ {\mathfrak{f}(\varepsilon)\vert M}}}\bigoplus_{d\vert (N/M)}\alpha_d(S_{k,k}(\Gamma_1(M),\tilde{\varepsilon})_{\text{new}}),$$
where $\tilde{\varepsilon}$ is the restriction of $\varepsilon$ to $\Gamma_1(M)$. 

\end{enumerate}
\end{prop}
\begin{proof}
For the first statement, we mostly adapt the proof in \cite{MR2738153} for congruence subgroups of $\SL_2(\IZ)$. We first count the number of cusps for a congruent subgroup $\Gamma\leq\SL_2(\calO_F)$. There is a bijection 
\begin{align*}
\Gamma\backslash\SL_2(\calO_F)/ B_0&\leftrightarrow \{\Gamma\text{-cusps}\}\\
\gamma&\mapsto \gamma B_0\gamma^{-1}
\end{align*}
where $B_0$ is the standard Borel subgroup of upper triangular matrices. If the class number is one, there is one $\SL_2(\calO_F)$-conjugacy class of Borel subgroups. The cosets $\SL_2(\calO_F)/B_0$ can be identified with projective space $\IP_1(F)$ and the number of cusps equals the number of $\Gamma$-orbits on $\IP_1(F)$. We first consider for subgroups $H\subseteq (\calO_F/N)^\times$ the congruence subgroups 
$$\Gamma_H(N)=\{\begin{pmatrix}a&b\\c&d\end{pmatrix}\in \SL_2(\calO_F)\vert c\equiv 0(N) \text{ and }a,d \in H\}$$ 
and denote by $M_{k,k}(\Gamma_H(N))\subseteq M_{k,k}(\Gamma_1(N))$ the corresponding space of forms. Mapping $\gamma=\begin{pmatrix}a&b\\c&d\end{pmatrix}\mapsto (c,d)$ induces a well-defined map from the cosets $\Gamma_H(N)\backslash\SL_2(\calO_F)$ to $\IP_H(N)=\{(c,d)\in (\calO_F/N)^\times\vert (c,d,N)=1\}/\sim$, with the equivalence relation $(c,d)\sim (c',d') \Leftrightarrow c'=hc,d'=hd$ for some $h\in H$. It is easy to check this map is a bijection and therefore the number of cusps for $\Gamma_H(N)$ is 
\begin{align*}\vert \IP_H(N)\vert&=\frac{\vert (\calO_F/N)^\times\vert}{ \vert H\vert}\cdot\vert \IP_1(\calO_F/N)\vert\\
&=\frac{\vert (\calO_F/N)^\times\vert}{ \vert H\vert}\calN(N)\prod_{\ip\vert N}(1+\calN(\ip)^{-1}).
\end{align*}
Moreover, letting $N_d$ denote the least common multiple of $d$ and $N/d$ and letting $H_d$ denote the image modulo $N_d$ of $H$, it follows easily from the definitions that 
\begin{equation}\label{eqn:cusps}\vert \IP_H(N)\vert=\sum_{d\vert N}\frac{\vert(\calO_F/d)^\times\vert\cdot \vert(\calO_F/(N/d))^\times\vert}{\vert H_d\vert}.
\end{equation}
 Let now $\varepsilon:(\calO_F/N)^\times\to \IC^\times$ a Dirichlet character of order $n$, and set $H=\ker(\varepsilon)$. For simplicity first assume that $(\calO_F/N)^\times/H$ is cyclic. Then $M_{k,k}(\Gamma_H(N))=\bigoplus_{m=0}^{n-1}M_{k,k}(\Gamma_1(N),\varepsilon^m)$,
and since the dimensions only depend on the Galois conjugacy class of the Nebentypus, we have 
$\dim M_{k,k}(\Gamma_{\ker(\varepsilon)}(N))=\sum_{\delta \vert n}\varphi(\delta)\dim M_{k,k}(\Gamma_1(N),\varepsilon^{n/\delta})$.
Applying M\"obius inversion yields the formula
\begin{equation}\label{eqn:nebentypuscyclic}\dim M_{k,k}(\Gamma_1(N),\varepsilon)=1/\varphi(n) \sum_{\delta\vert n}\mu(\delta)\dim M_{k,k}(\Gamma_{\ker(\varepsilon^{\delta})}(N))
\end{equation}
In general, letting $\varepsilon_1,\ldots, \varepsilon_r$ denote generators of the group of Dirichlet characters for $\calO_K$ modulo $N$, and defining positive integers $\vec{n}=(n_1,\ldots,n_r)$ so that if $\varepsilon=\varepsilon_1^{k_1}\cdots \varepsilon_r^{k_r}$ then $\varepsilon_i$ has order $n_i\cdot k_i$, one gets after successively applying M\"obius inversion that 
\begin{equation}\label{eqn:nebentypus}\dim M_{k,k}(\Gamma_1(N),\varepsilon)=1/\varphi(\vec{n}) \sum_{\vec{\delta}\vert \vec{n}}\mu(\vec{\delta})\dim M_{k,k}(\Gamma_{\ker(\varepsilon^{\vec{\delta}})}(N)),
\end{equation}
where $\varphi(\vec{n})=\varphi(n_1)\cdots \varphi(n_r)$ and $\mu(\vec{\delta})=\mu(\delta_1)\cdots \mu(\delta_r)$. 
The same reasoning and formulas are also valid considering only spaces of cusp forms or Eisenstein series.

We now show that the dimension of the space of Eisenstein series of level $\Gamma_1(N)$, weight $(k,k)$ and Nebentypus $\varepsilon$ such that $\varepsilon(-1)=(-1)^k$ equals 
$$\nu_\infty(\varepsilon)=-\delta_{k,2}\cdot\delta_{\mathfrak{f},1}+\sum_{\substack{d\vert N\\ (d,N/d)\vert N/\mathfrak{f}}}\vert (\calO_F/(d,N/d))^\times\vert.$$
First we remark that the Borel-Serre compactification consists here of adding a (filled in) torus at each cusp $c$, and each contributes one dimension to the homology of the boundary coming from the longitudinal loop $\gamma_c$. Working with homology, by Poincar\'e duality the space of Eisenstein series has dimension the rank of $\image(H_1(\partial Y(\Gamma_0(N)),\underline{V})\to H_1(Y(\Gamma_0(N)),\underline{V}))$, with relevant coefficients. The image is generated by the cycles coming from each cusp, provided $\varepsilon(-1)=(-1)^k$, and the only boundary relation in $H_1(Y(\Gamma_0(N)),\underline{V})$ comes from $\sum_{c}[\gamma_c]\in B_1(Y(\Gamma_0(N)),\IZ)$, which cuts down the dimension by $1$ in trivial weight and conductor. From this and equations \ref{eqn:cusps} and \ref{eqn:nebentypus} the claim reduces to showing that for all $d\vert N$ and fixed $\varepsilon$: 
$$\sum_{\vec{\delta}\vert \vec{n}}\mu(\vec{\delta})\frac{\vert(\calO_F/d)^\times\vert\cdot \vert(\calO_F/(Nd^{-1}))^\times\vert}{\vert \ker(\varepsilon^{\vec{\delta}})_d\vert}
=\begin{cases}
\varphi(\vec{n})\cdot \vert (\calO_F/(d,N/d))^\times\vert& \text{if }\mathfrak{f} \vert N_d\\
0& \text{else.}\\

\end{cases} $$
If $\mathfrak{f}\vert N_d$, then we may simply write $\vert  \ker(\varepsilon^{\vec{\delta}})_d\vert=(\vec{\delta}/\vec{n})\cdot\vert (\calO_F/N_d)^\times\vert$ and the sum becomes
\begin{align*}\sum_{\vec{\delta}\vert \vec{n}}\mu(\vec{\delta})(\vec{n}/\vec{\delta})\frac{\vert(\calO_F/d)^\times\vert\cdot \vert(\calO_F/Nd^{-1})^\times\vert}{\vert (\calO_F/N_d)^\times \vert}&=\varphi(\vec{n}) \frac{\vert(\calO_F/d)^\times\vert\cdot \vert(\calO_F/Nd^{-1})^\times\vert}{\vert (\calO_F/N_d)^\times \vert}\\
&=\varphi(\vec{n})\cdot \vert (\calO_F/(d,N/d))^\times\vert.
\end{align*}

Now assume there exists a prime $\ip\vert p$ with $v_{\ip}(\mathfrak{f})>v_{\ip}(N_d)$. Since $\ip\vert N$ we have that $v_\ip(\mathfrak{f})\geq 2$. Denoting by $F_d(\varepsilon^{\vec{\delta}})$ the kernel of the reduction map $\ker(\varepsilon^{\vec{\delta}})\to\ker(\varepsilon^{\vec{\delta}})_d$ modulo $N_d$, one checks for $\vec{\delta}$ coprime to $p$ that $\vert F_d(\varepsilon^{\vec{\delta}\cdot p})\vert=p\cdot\vert F_d(\varepsilon^{\vec{\delta}})\vert$, where $\vec{\delta}\cdot p$ denotes $(\delta_1\cdot 1,\ldots,\delta_{i\ip}\cdot p,\ldots,\delta_r\cdot 1)$ with index $i\ip$ corresponding to the generator of the $p$-power order part of $(\calO_F/\ip^{v_\ip(N)})^\times$. Moreover, $\vert \ker(\varepsilon^{\vec{\delta}\cdot p})\vert=p\cdot\vert \ker (\varepsilon^{\vec{\delta}})\vert$ so that $\vert  \ker(\varepsilon^{\vec{\delta}\cdot p})_d\vert=\vert  \ker(\varepsilon^{\vec{\delta}})_d\vert$ for $\vec{\delta}$ coprime to $p$. This leads to cancellation in the sum of the terms in $\mu(\vec{\delta})$ and in $\mu(\vec{\delta}\cdot p)=-\mu(\vec{\delta})$ and the claim follows. Finally, the product representation of $\sum\vert (\calO_F/(d,N/d))^\times\vert$ is helpful for computations and straightforward but slightly tedious. It can be obtained adapting the proof in \cite[Lemma 3.8]{MR2738153} for $\SL_2(\IZ)$. For the statement on oldforms, this is done in \cite{MR0299559}. The proof over $\IQ$ generalizes to $F$. 

\end{proof}
\subsection{Computational method} 
We carried through computations for $F=\mathbb{Q}(\theta)$ with $\theta=\sqrt{-2}$ and $\mathcal{O}_F=\IZ[\theta]$. We write $p=3=\pi\overline{\pi}$ for the factorization of the split prime $3$ in $\calO_F=\IZ(\theta)$, where $\pi=(1+\theta)$ and fix these notations for the rest of the paper. We note that our code can easily be adapted to work for the other Euclidean imaginary quadratic fields and we expect similar results.\\
We computed the relevant spaces of Bianchi modular forms with nebentypus by computing the group cohomology $H^2(\Gamma_0(N),V_{k,k}(\IC^2)\otimes\varepsilon)$, recalling the notation $V_{k,k}(\IC^2)=\Sym^k(\IC^2)\otimes \overline{\Sym^k(\IC^2)}$.
Even for trivial coefficients $(k=0)$, which we will mostly consider, our approach is less expensive than working with presentations of the Bianchi group and computing the abelianizations of congruent subgroups. Instead we use Shapiro's Lemma and compute 
$$H^2(\SL_2(\calO_F), \Ind_{\Gamma_0(N)}^{\SL_2}(V_{k,k}(\IC^2)\otimes\varepsilon)).$$
By Poincar\'e duality, this has the same dimension as $ H^1(\SL_2(\calO_F), \Ind_{\Gamma_0(N)}^{\SL_2}(V_{k,k}(\IC^2)\otimes\varepsilon))\cong H^1(\Gamma_0(N),V_{k,k}\otimes\varepsilon)$, and therefore computes the dimension of the space of cohomological automorphic forms of interest. \par

Our computations are based on the Magma implementation of an algorithm by Alexander Rahm and Haluk \c{S}eng\"{u}n that computes the dimensions of $H^2(\Gamma_0(N), V_{k,k}(\IC^2))$, which we then adapted to work more generally for forms with nebentypus. All of our computations were carried out with the Magma computational algebra system (\cite{MR1484478}). The approach of Rahm and \c{S}eng\"{u}n is based on reduction theory as found in \cite{MR728453}: one constructs a two-dimensional contractible CW-complex inside $\mathbb{H}_3$ that is a fundamental domain for the action of $\SL_2(\calO_F)$. One then uses this combinatorial data to compute the cohomology via the equivariant cohomlogy spectral sequence.
 Explicitly, let $\Gamma:=\SL_2(\mathcal{O}_F)$ and let $M$ be a finitely generated right $\Gamma$-module over some ring containing $\mathcal{O}_F$. The spectral sequence yields:
$$E_1^{p,q}=\bigoplus_{\sigma \in \Sigma_p}H^q(\Gamma_\sigma,M)\Rightarrow H^{p+q}(\Gamma,M),$$
where $\Sigma_p$ consists of representatives for the $\Gamma$-conjugacy classes of $p$-cells and where $\Gamma_\sigma$ denotes the respective stabilizers. The fundamental cellular domain for the action of $\Gamma$ on the hyperbolic $3$-space, which was supplied to us by A. Rahm, consists of $2$ two-cells and $6$ edges. The stabilizers of the two-cells are both $\pm\Id\cong C_2$. The stabilizers of the edges are given by:
\begin{align*}
\Gamma_1&=<\begin{pmatrix}0&-1\\1&0\end{pmatrix}>\cong C_4 &
\Gamma_2&\cong C_2\\
\Gamma_3&=<\begin{pmatrix}-\theta&1\\1&\theta\end{pmatrix}>\cong C_4&
\Gamma_4&=<\begin{pmatrix}-1&-\theta\\-\theta&1\end{pmatrix}>\cong C_4\\
\Gamma_5&=<\begin{pmatrix}1&1\\-1&0\end{pmatrix}>\cong C_6&
\Gamma_6&=<\begin{pmatrix}1+\theta&-1+\theta\\-1&-\theta\end{pmatrix}>\cong C_6\\
\end{align*}
The differential $d_1^{1,0}:E_1^{1,0}\to E_1^{2,0}$ is given by:
\begin{align*}
d_1^{1,0}:M^{\Gamma_1}\oplus\ldots\oplus M^{\Gamma_6}&\to M^{<-\Id>}\oplus M^{<-\Id>}\\
(m_1,\ldots,m_6)&\mapsto (-m_1+m_2-m_4+m_5, -m_2\cdot G^{-1}-m_3+m_4+m_6),
\end{align*}
where $G=\begin{pmatrix}-1&\theta\\0&-1\end{pmatrix}$. \\

Over $\calO_F[1/6]$ the stabilizers do not contribute for $q>0$ and the spectral sequence is concentrated on $q=0$ and stabilizes on the $E_2$-page. Thus, computing the dimension of $E_2^{2,0}\cong E_1^{2,0}/d_1^{1,0}(E_1^{1,0})$ yields the dimension of $H^2(\Gamma, M)$. For more details on this approach we refer to \cite[Section 5]{MR2859903} and \cite{MR3091734,Rahm:2017aa}.\\
The Hecke eigenvalue computations in trivial weight were carried out by acting for $M=\Ind_{\Gamma_0(N)}^{\SL_2}{\varepsilon}$ on $E_1^{2,0}$. The Hecke action then preserves the image of the differential $d_1^{1,0}(E_1^{1,0})$ and descends to cohomology. The action was implemented for Euclidean imaginary quadratic fields via the approach based on using Heilbronn matrices to compute the action of the Hecke operator $T_\il$ on Manin symbols for each prime $\il$ not dividing the level. This is described in A. Mohamed's thesis \cite[Chapters 3-4]{Mohamed:thesis} in our setting; for the classical setting over $\IQ$ we refer the reader to J. Cremona's book \cite[Chapter 2]{MR1628193}. \\
In all of our computations, we noted that the dimensions and Hecke eigenvalues obtained were consistent with existing computations for trivial and quadratic nebentypus, as well as with lower bounds for dimensions coming from Eisenstein series and oldforms as in Proposition \ref{prop:eisold}. The dimensions for trivial nebentypus can be accessed as part of the already existing large-scale computations on the LMFDB \cite{lmfdb:Bianchi}.
Finally, in the Hecke eigenvalue computations, precisely the computed number of cuspidal eigenforms satisfy the Ramanujan bounds, whereas the remaining eigensystems coming from Eisenstein series violate the bounds.

\subsection{Examples}
We start by considering the example of tame level $N=(3-2\theta)$ which appears in \cite[Theorem 1.1]{MR2461903} and compute with spaces up to level $\Gamma_0(N)\cap\Gamma_1(9)\subset\SL_2(\calO_F)$. We record here the dimensions of the new subspaces obtained after subtracting the dimensions of spaces of Eisenstein series and oldforms as in Proposition \ref{prop:eisold}. We omit levels which do not have any newforms and mark in boldface the dimensions of spaces which could give rise to ordinary forms on the Hida family, according to Lemma \ref{lemma:ordinary}. 

\begin{table}[h!]
    \centering
  \begin{tabular}{|L|L|L|L|L|L|L|L|L|L|}
\hline
\text{Level }\backslash\text{Nebentypus conductor}&0&\pi&\overline{\pi}&\pi^2&\overline{\pi}^2&3&3\pi&3\overline{\pi}&9\\
\hline\hline
\Gamma_0(N)\cap\Gamma_1(\pi)&\bf{1}&0&-&-&-&-&-&-&-\\
\Gamma_0(N)\cap\Gamma_1(3)&0&0&0&-&-&0&-&-&-\\
\Gamma_0(N)\cap\Gamma_1(\pi^2)&1&0&-&0&-&-&-&-&-\\
\Gamma_0(N)\cap\Gamma_1(\overline{\pi}^2)&0&-&0&-&0&-&-&-&-\\
\Gamma_0(N)\cap\Gamma_1(3\pi)&0&0&0&\bf{8}&-&0&0&-&-\\
\Gamma_0(N)\cap\Gamma_1(3\overline{\pi})&3&0&0&0&\bf{6}&0&-&0&-\\
\Gamma_1(9)&0&0&0&0&0&4&0&0&20\\
\Gamma_0(N)\cap\Gamma_1(9)&6&0&0&8&6&0&0&0&0\\

\hline
\end{tabular}
 
    \caption{Dimensions of new subspace for $N=3-2\theta$}
\label{table:tabone}
\end{table}

We are interested in the family passing through the newform $\Pi$ at level $\Gamma_0(N\pi)$ which via modularity should correspond to the elliptic curve over $F$:
$$E\equiv y^2+\theta xy+y=x^3+(\theta-1)x^2-\theta x.$$
In this particular case presumably modularity can be checked via the Faltings-Serre method. The dimensions in boldface in Table \ref{table:tabone} could a priori still contain cuspidal Bianchi modular forms which arise via specialization from the Hida family. We want to eliminate this possibility via a computation of Hecke eigenvalues, showing that the congruences betweeen Hecke eigenvalues expected if this were the case do not hold. \\
We list a representative of each Galois orbit of Hecke eigenvalue and we limit ourselves to primes not dividing the level. Moreover, since $\Pi$ is actually modular by a congruence subgroup of $\PGL_2(\calO_F)$, i.e. is furthermore fixed by the action of $J:=\tiny{\begin{pmatrix}
-1&0\\0&1
\end{pmatrix}}$, it suffices to consider cuspforms fixed under the involution induced by $J$ on the relevant spaces and we only list those eigenvalues.
\begin{table}[h!]
    \centering
  \begin{tabular}{|L|L|L|L|L|L|L|L|L|L|}
\hline
\text{Level, conduct. }\mathfrak{f}(\varepsilon)&\sharp (G_\IQ\text{-orbit})&\theta&1-\theta&3+\theta&3-\theta&3+2\theta&1+3\theta&1-3\theta\\
\hline\hline
\Gamma_0(N\pi), 1&1&-2&-2&-2&-4&-6&0&-4\\
\hline
\Gamma_0(N)\cap\Gamma_1(3\overline{\pi}),\overline{\pi}^2&2&-2\zeta_3&-&-5\zeta_3&-4&6\zeta_3&0&2\zeta_3\\
&2&\zeta_3&-&4\zeta_3&5&0&-6\zeta_3&-7\zeta_3\\
&2&0&-&-3\zeta_3&0&-6\zeta_3&8\zeta_3&2\zeta_3\\
\hline
\Gamma_0(N)\cap\Gamma_1(3\pi),\pi^2&6&\IQ(f_1)&-&\IQ(f_2)&\IQ(f_3)&\IQ(f_4)&\IQ(f_5)&\IQ(f_6)\\
\hline
\end{tabular}
 
    \caption{Hecke eigenvalues for $N=3-2\theta$}
\label{table:tabtwo}
\end{table}

Here $\zeta_3$ denotes a primitive third root of unity. At level $\Gamma_0(N)\cap\Gamma_1(3\pi)$ and Nebentypus conductor $\pi^2$ we listed only the fields of definition for the Hecke eigenvalues in Table \ref{table:tabtwo}; their minimal polynomials $f_1(x),\ldots,f_6(x)$ have degrees dividing $6$. It turns out that the polynomials $f_1, \ldots f_6$ all have roots congruent to the respective eigenvalue at level $\Gamma_0(N\pi)$. Extending the computation to the $28$ first $T_\il$-eigenvalues we observe similarly congruences between the Galois orbit at level $\Gamma_0(N)\cap\Gamma_1(3\pi)$ and level $\Gamma_0(N\pi)$ at a prime above $3$, so that here a congruence cannot be excluded but seems likely. We thus were forced to extend the computations to levels $\Gamma_0(N)\cap\Gamma_1(3\pi^2)$ and Nebentypus conductor $\mathfrak{f}(\varepsilon)\in\{\pi^3, 3\pi^2\}$ as well as level $\Gamma_0(N)\cap\Gamma_1(\pi^3)$ and Nebentypus conductor $\pi^3$ to try to obtain the desired finiteness result. We then observed that these spaces consisted exclusively of Eisenstein series, which sufficed for our purposes-see the proof of Proposition \ref{prop:CMrect} below.\par
We record our conclusion, rectifying Calegari and Mazur's argument in \cite[Theorem 1.1]{MR2461903}, which only excludes that the support of the Hida family contains the diagonal. This fails to account for the infinitely many possibilities for classical points with non-parallel nebentypus characters. Hence their argument only shows there are finitely many \emph{crystalline} classical points on the family. 
\begin{prop}\label{prop:CMrect}
The component of the ordinary $3$-adic Hecke algebra $\IT$ of tame level $N=(3-2\theta)$ passing through the newform $\Pi$ at level $\Gamma_0(N\pi)$ contains only finitely many classical points. 
\end{prop}
\begin{proof}
Assume there are infinitely many classical points. Since we know from computations (e.g., \cite[Lemma 8.8]{MR2461903}) that the specialization of $\IT$ in weight $(4,4)$ cannot correspond to a classical automorphic form of level $\Gamma_0(N)$, it follows that the support of $\IT$ cannot contain the diagonal.
Now $\Pi$ is the unique newform at level $\Gamma_0(N\pi)$. Thus any system of Hecke eigenvalues on that same component, viewed via a fixed embedding $\overline{\IQ}\to\overline{\IQ}_3$ as valued in a finite extension $E/\mathbb{Q}_3$, will be congruent modulo some prime above $3$ in $E$. The computations in Tables \ref{table:tabone} and \ref{table:tabtwo} show that there is no ordinary cuspidal eigenform with Hecke eigenvalues congruent to $\Pi$, level dividing $\Gamma_0(N)\cap\Gamma_1(9)\subset\PGL_2(\calO_F)$ and trivial weight except possibly at level $\Gamma_0(N)\cap\Gamma_1(3\pi)$ and nebentype of conductor $\pi^2$. Indeed, in the remaining cases to exclude one sees in Table \ref{table:tabtwo} that the three Galois orbits of newforms at level $\Gamma_0(N)\cap\Gamma_1(3\overline{\pi})$ do not have congruent $T_\il$-eigenvalues to $\Pi$ for either $\il=(1-3\theta)$ or $\il=(3-\theta)$. This does not quite suffice to conclude via the rigidity results or Corollary \ref{cor:checktwists} because of the possible congruence with a level $\Gamma_0(N)\cap\Gamma_1(3\pi)$ and conductor $\pi^2$ newform. To establish the finiteness result via Corollary \ref{cor:checktwists} it would however suffice to exclude that there is an additional newform of level dividing $\Gamma_0(N)\cap\Gamma_1(27)$ on the ordinary family. \par 
In fact, a more careful analysis further reduces the necessary computations. More precisely, using the rigidity results in Lemma \ref{lemma:ktozero}, what remains to be done is exclude the possibility that the support of the Hida family in weight space contains a component of the form $\Spec(\Lambda/((X+1)^N-(Y+1)))$ for some $N\in \IZ_3$, after possibly switching the roles of $X$ and $Y$ and having normalized the $\Lambda$-module structure so that $\Pi$ lies above the origin. Choosing the $X$-coordinate in the completed group ring to correspond to the $\pi$-direction, we can further narrow down the possibilities to a component of the form $\Spec(\Lambda/((X+1)^M-(Y+1)))$ for $M\in 3\mathbb{Z}_3$. Indeed, otherwise our computations at level dividing $\Gamma_0(N)\cap\Gamma_1(9)$ would exhibit a congruent Hecke eigensystem with Nebentypus conductor at least as divisible by $\overline{\pi}$ as by $\pi$, which is not the case. \par
 Finally, to exclude a component of the form $\Spec(\Lambda/((X+1)^M-(Y+1)))$ for $M\in 3\mathbb{Z}_3$ on the support of $\IT$, observe that specializing at $(\zeta_{9}-1,0)$ or $(\zeta_{9}-1,\xi_3-1)$ for some primitive roots of unity $\zeta_{9},\xi_3$ of orders $9$ and $3$ respectively, one should in that case recover an additional classical cuspidal eigensystem on the ordinary family. Such a specialization would imply ordinary cuspidal newforms at levels $\Gamma_0(N)\cap\Gamma_1(3\pi^2)$ and Nebentypus conductor $\mathfrak{f}(\varepsilon)\in\{\pi^3, 3\pi^2\}$ or at level $\Gamma_0(N)\cap\Gamma_1(\pi^3)$ and Nebentypus conductor $\pi^3$ as remaining possibilities. Recall also we only really need to consider
the spaces fixed by $J=\tiny{\begin{pmatrix}
-1&0\\0&1
\end{pmatrix}}$. A computation of those spaces yields: 
 \begin{table}[h!]
    \centering
\begin{tabular}{|L|L|L|}
\hline

\text{Level } &\text{Nebentypus conductor}& \text{dimension of cuspidal space}\\
\hline\hline
\Gamma_0(N)\cap\Gamma_1(\pi^3)&\pi^3&0\\
\Gamma_0(N)\cap\Gamma_1(3\pi^2)&\pi^3&0\\
\Gamma_0(N)\cap\Gamma_1(3\pi^2)&3\pi^2&0\\
\hline
\end{tabular}
    \caption{Excluding remaining cuspidal spaces for $N=3-2\theta$}
\label{table:tabfix}
\end{table}
\par
 We thus deduce the desired finiteness result.
 
\end{proof}

We use this strategy to study ordinary $3$-adic families passing through Bianchi newforms in trivial weight for a whole range of tame levels:

\begin{theorem}\label{theorem:maincomp}
Let $F=\IQ(\sqrt{-2})$. We examined all the (nearly) ordinary $3$-adic Hecke algebras specializing in trivial weight to a cuspidal newform of level $\Gamma_0(I)$ for ideals $I\subset\calO_F$ of norm $\calN(I)\leq 200$:
\end{theorem}
\begin{enumerate}
    \item The $3$-adic families of tame levels $N=(3-2\sqrt{-2})$, $N=(3+\sqrt{-2})$ and $N=(6+\sqrt{-2})$, specializing in trivial weight to newforms of levels $\Gamma_0(7+\sqrt{-2})$, $\Gamma_0(9+3\sqrt{-2})$ and $\Gamma_0(4+7\sqrt{-2})$ respectively, interpolate finitely many classical automorphic forms.
    \item The $3$-adic family of tame level $N=2\sqrt{-2}$ specializing in trivial weight to a newform of level $\Gamma_0(6\sqrt{-2})$ is a base-change family and in particular interpolates infinitely many automorphic forms base-changed from $\IQ$. 
\end{enumerate}
\begin{proof}
We start by listing the newforms appearing at levels $\Gamma_0(I)\subset \SL_2(\calO_F)$ for  ideals $I\subset\calO_F$ of norm $\mathcal{N}(I)\leq 200$ divisible by $3$  in Table \ref{table:tabthree}. This already appears in \cite[Table 3.3.1.]{MR743014} and, using similar notations, the second column marks rational newforms by a boldface $\bf{1}$ and indicates the sign of the forms under the involution induced by $J$. The newforms and some additional information can also be found as part of the extensive databases on the LMFDB Bianchi modular forms webpage \cite{lmfdb:Bianchi}. The newforms invariant under $J$ are exactly the ones contributing to the cohomology of congruence subgroups of $\PGL_2(\calO_F)$, and we only need to consider those. We added the LMFDB \cite{lmfdb} labels for these newforms in Table \ref{table:tabthree} for the reader's convenience.
\begin{table}[h!]
    \centering
\begin{tabular}{|L|L|L|L|L|}
\hline

\text{level } &\text{newforms: sign + (label), - }&\text{level factorisation}& \text{ordinary at }3\\
\hline\hline
\Gamma_0(6)&0,\mathbf{1}& \theta^2\cdot 3& \text{yes}\\
\Gamma_0(7+\theta)&\mathbf{1}\text{ (\cite[\href{http://www.lmfdb.org/ModularForm/GL2/ImaginaryQuadratic/2.0.8.1/51.1/a/}{51.1-a}]{lmfdb})},0& (3-2\theta)\cdot (1+\theta)& \text{yes}\\
\Gamma_0(2+5\theta)&\mathbf{1}\text{ (\cite[\href{http://www.lmfdb.org/ModularForm/GL2/ImaginaryQuadratic/2.0.8.1/54.1/a/}{54.1-a}]{lmfdb})},\mathbf{1}& (\theta)\cdot (1+\theta)^3& \text{no}\\
\Gamma_0(8+\theta)&0,\mathbf{1}& \theta\cdot(3-\theta)\cdot (1-\theta)&\text{yes}\\
\Gamma_0(6\theta)&\mathbf{1}\text{ (\cite[\href{http://www.lmfdb.org/ModularForm/GL2/ImaginaryQuadratic/2.0.8.1/72.2/a/}{72.2-a}]{lmfdb})},0& \theta^3\cdot 3& \text{yes}\\
\Gamma_0(5+5\theta)&0,\mathbf{1}& 5\cdot (1+\theta)&\text{yes}\\
\Gamma_0(9+3\theta)&\mathbf{1}\text{ (\cite[\href{http://www.lmfdb.org/ModularForm/GL2/ImaginaryQuadratic/2.0.8.1/99.3/a/}{99.3-a}]{lmfdb})},0& (3+\theta)\cdot 3& \text{yes}\\
\Gamma_0(6+6\theta)&\mathbf{1}\text{ (\cite[\href{http://www.lmfdb.org/ModularForm/GL2/ImaginaryQuadratic/2.0.8.1/108.2/a/}{108.2-a}]{lmfdb})},0& \theta^2\cdot (1+\theta)^2\cdot(1-\theta)&\text{no}\\
\Gamma_0(4+7\theta)&\mathbf{1}\text{ (\cite[\href{http://www.lmfdb.org/ModularForm/GL2/ImaginaryQuadratic/2.0.8.1/114.1/a/}{114.1-a}]{lmfdb})},\mathbf{1}& \theta\cdot(1-3\theta)\cdot (1+\theta)&\text{yes}\\
\Gamma_0(11+\theta)&0,\mathbf{1}& (3-4\theta)\cdot (1+\theta)&\text{yes}\\
\Gamma_0(11+2\theta)&0,\mathbf{1}+\mathbf{1}& (5-3\theta)\cdot (1+\theta)&\text{yes}\\
\Gamma_0(10+7\theta)&\mathbf{1}\text{ (\cite[\href{http://www.lmfdb.org/ModularForm/GL2/ImaginaryQuadratic/2.0.8.1/198.1/a/}{198.1-a}]{lmfdb})},0& \theta \cdot(3+\theta)\cdot (1+\theta)^2&\text{no}\\
\hline
\end{tabular}
    \caption{Newforms for level $\Gamma_0(I)$ when $3\vert\mathcal{N}(I)\leq 200$}
\label{table:tabthree}
\end{table}

 Among these, one checks that the ordinary newform at level $\Gamma_0(6\theta)$ is base-change. Proposition \ref{prop:CMrect} already deals with the level $\Gamma_0(7+\theta)$. It remains therefore to establish finiteness for tame levels $N=(3+\theta)$ and $N=(6+\theta)$. \\
We first deal with the case of $N=(3+\theta)$. Again a weight $(4,4)$ computation (see \cite[Lemma 8.8]{MR2461903}) excludes the diagonal case. In trivial weight, for our purposes it suffices to consider levels and conductors that contribute to spaces that may allow for automorphic forms ordinary above $3$ in the family.\par
We list in Table \ref{table:tabfour} the computed dimensions of the new subspaces for such levels dividing $\Gamma_0(N)\cap\Gamma_1(9)$ and conductors of nebentypi, whenever non-trivial. 
\begin{table}[h!]
    \centering
\begin{tabular}{|L|L|L|}
\hline

\text{Level } &\text{Nebentypus conductor}& \text{dimension}\\
\hline\hline
\Gamma_1(9)&3&4\\
&9&20\\
\Gamma_0(N)\cap\Gamma_1(3)&0&1\\
\Gamma_0(N)\cap\Gamma_1(\pi^2)&\pi^2&2\\
\Gamma_0(N)\cap\Gamma_1(9)&\pi^2&2\\
\hline
\end{tabular}
    \caption{Dimensions of new subspace for $N=3+\theta$}
\label{table:tabfour}
\end{table}
In this case, we note that only the two newforms at level $\Gamma_0(N)\cap\Gamma_1(\pi^2)$ are ordinary at both places above $3$, however a computation shows that they both lie in the minus subspace under $J$. We therefore obtain the desired finiteness of classical points without needing to compute Hecke eigenvalues.\par
We now consider tame level $N=(6+\theta)=(\theta)\cdot(1-3\theta)$.
\begin{table}[h!]
    \centering
\begin{tabular}{|L|L|L|}
\hline

\text{Level } &\text{Nebentypus conductor}& \text{dimension}\\
\hline\hline
\Gamma_0(N)&0&1\\
\Gamma_1(9)&3&4\\
&9&20\\
\Gamma_0(\theta)\cap\Gamma_1(9)&9&4\\
\Gamma_0(1-3\theta)\cap\Gamma_1(3)&0&1\\
\Gamma_0(N)\cap\Gamma_1(\pi)&0&\bf{2}\\
\Gamma_0(1-3\theta)\cap\Gamma_1(3\pi)&\pi^2&2\\
\Gamma_0(1-3\theta)\cap\Gamma_1(3\overline{\pi})&\overline{\pi}^2&2\\
\Gamma_0(N)\cap\Gamma_1(3\pi)&\pi^2&\bf{4}\\
\Gamma_0(N)\cap\Gamma_1(3\overline{\pi})&\overline{\pi}^2&\bf{2}\\
\hline
\end{tabular}
    \caption{Dimensions of new subspace for $N=\theta(1-3\theta)$}
\label{table:tabfive}\end{table}
  
Table \ref{table:tabfive} lists the dimensions of new subspaces in trivial weight for levels dividing $\Gamma_0(N)\cap\Gamma_1(9)$ and conductors that contribute to spaces that may contain automorphic forms ordinary above $3$ in the family, when these dimensions are non-zero. Again a computation in weight $(4,4)$ found no cuspforms of level $\Gamma_0(N)$.
Table \ref{table:tabsix} exhibits Galois orbits of Hecke eigenvalues for the eigensystems marked in boldface occuring in the plus subspace which remain to be excluded from the family.
\begin{table}[h!]
    \centering
  \begin{tabular}{|L|L|L|L|L|L|L|L|L|L|}
\hline
\text{Level, conduct. }\mathfrak{f}(\varepsilon)&\sharp (G_\IQ\text{-orbit})&3-\theta&3+\theta&3-2\theta&3+2\theta&1+3\theta\\
\hline\hline
\Gamma_0(N\pi), 1&1&4&4&-6&-2&0\\
\hline
\Gamma_0(N)\cap\Gamma_1(3\overline{\pi}),\overline{\pi}^2&2&-\zeta_3&2&-\zeta_3&0&-3\zeta_3\\
\hline
\Gamma_0(N)\cap\Gamma_1(3\pi),\pi^2&2&0&0&-3\zeta_3&3\zeta_3&-4\zeta_3\\
\hline
\end{tabular}

    \caption{Hecke eigenvalues for $N=\theta(1-3\theta)$}
\label{table:tabsix}
\end{table}
In this example, again one sees that there are no congruences, and thus none of the considered eigensystems occurs on the component of the Hecke algebra of the unique newform at level $\Gamma_0(N\pi)$. The desired finiteness result follows, concluding the proof. 

\end{proof}

%%%%%%%%%%%%%%%%%%%%%%%%%%%%%%%%%%%%
\section{Acknowledgements}
The author would like to extend special gratitude to Frank Calegari for pointing us toward this project and for his support throughout. The author is very greatful to Alexander Rahm and Haluk \c{S}eng\"{u}n for their incredible support on the computational aspects of the paper and for supplying much of the code that underlies the computations. We would also like to thank the referee for a careful reading of this manuscript and for excellent suggestions on improving this text. Finally, we would like to thank Harald Grobner, Joel Specter and Gabor Wiese for helpful conversations on some of the content of this paper. 
\nocite{}

\bibliography{hidapaperbiblio.bib}

\begin{thebibliography}{10}

\bibitem{2018arXiv181209999A}
P.~B. {Allen}, F.~{Calegari}, A.~{Caraiani}, T.~{Gee}, D.~{Helm}, B.~V. {Le
  Hung}, J.~{Newton}, P.~{Scholze}, R.~{Taylor}, and J.~A. {Thorne}.
\newblock {Potential automorphy over CM fields}.
\newblock {\em arXiv e-prints}, page arXiv:1812.09999, Dec. 2018.

\bibitem{MR2396124}
A.~Ash, D.~Pollack, and G.~Stevens.
\newblock Rigidity of {$p$}-adic cohomology classes of congruence subgroups of
  {${\rm GL}(n,\mathbb{Z})$}.
\newblock {\em Proc. Lond. Math. Soc. (3)}, 96(2):367--388, 2008.

\bibitem{MR1484478}
W.~Bosma, J.~Cannon, and C.~Playoust.
\newblock The {M}agma algebra system. {I}. {T}he user language.
\newblock {\em J. Symbolic Comput.}, 24(3-4):235--265, 1997.
\newblock Computational algebra and number theory (London, 1993).

\bibitem{bygott1998modular}
J.~S. Bygott.
\newblock {\em Modular Forms and Modular Symbols over Imaginary Quadratic
  Fields JS Bygott PhD 1998}.
\newblock PhD thesis, University of Exeter, 1998.

\bibitem{MR2461903}
F.~Calegari and B.~Mazur.
\newblock Nearly ordinary {G}alois deformations over arbitrary number fields.
\newblock {\em J. Inst. Math. Jussieu}, 8(1):99--177, 2009.

\bibitem{MR1639612}
B.~Conrad, F.~Diamond, and R.~Taylor.
\newblock Modularity of certain potentially {B}arsotti-{T}ate {G}alois
  representations.
\newblock {\em J. Amer. Math. Soc.}, 12(2):521--567, 1999.

\bibitem{MR743014}
J.~E. Cremona.
\newblock Hyperbolic tessellations, modular symbols, and elliptic curves over
  complex quadratic fields.
\newblock {\em Compositio Math.}, 51(3):275--324, 1984.

\bibitem{MR1628193}
J.~E. Cremona.
\newblock {\em Algorithms for modular elliptic curves}.
\newblock Cambridge University Press, Cambridge, second edition, 1997.

\bibitem{MR892187}
G.~Harder.
\newblock Eisenstein cohomology of arithmetic groups. {T}he case {${\rm
  GL}_2$}.
\newblock {\em Invent. Math.}, 89(1):37--118, 1987.

\bibitem{MR3542488}
E.~Hellmann and B.~Schraen.
\newblock Density of potentially crystalline representations of fixed weight.
\newblock {\em Compos. Math.}, 152(8):1609--1647, 2016.

\bibitem{MR848685}
H.~Hida.
\newblock Galois representations into {${\rm GL}_2({\bf Z}_p[[X]])$} attached
  to ordinary cusp forms.
\newblock {\em Invent. Math.}, 85(3):545--613, 1986.

\bibitem{MR868300}
H.~Hida.
\newblock Iwasawa modules attached to congruences of cusp forms.
\newblock {\em Ann. Sci. {\'E}cole Norm. Sup. (4)}, 19(2):231--273, 1986.

\bibitem{MR960949}
H.~Hida.
\newblock On {$p$}-adic {H}ecke algebras for {${\rm GL}_2$} over totally real
  fields.
\newblock {\em Ann. of Math. (2)}, 128(2):295--384, 1988.

\bibitem{MR1203228}
H.~Hida.
\newblock {$p$}-ordinary cohomology groups for {${\rm SL}(2)$} over number
  fields.
\newblock {\em Duke Math. J.}, 69(2):259--314, 1993.

\bibitem{MR1313784}
H.~Hida.
\newblock {$p$}-adic ordinary {H}ecke algebras for {${\rm GL}(2)$}.
\newblock {\em Ann. Inst. Fourier (Grenoble)}, 44(5):1289--1322, 1994.

\bibitem{MR3220926}
H.~Hida.
\newblock Hecke fields of {H}ilbert modular analytic families.
\newblock In {\em Automorphic forms and related geometry: assessing the legacy
  of {I}. {I}. {P}iatetski-{S}hapiro}, volume 614 of {\em Contemp. Math.},
  pages 97--137. Amer. Math. Soc., Providence, RI, 2014.

\bibitem{MR0190146}
S.~Lang.
\newblock Division points on curves.
\newblock {\em Ann. Mat. Pura Appl. (4)}, 70:229--234, 1965.

\bibitem{lmfdb}
T.~{LMFDB Collaboration}.
\newblock The {L}-functions and {M}odular {F}orms {D}atabase.
\newblock \url{http://www.lmfdb.org}, 2013.
\newblock [Online].

\bibitem{lmfdb:Bianchi}
T.~{LMFDB Collaboration}.
\newblock {\itshape {The L-functions and Modular Forms Database, {\em Home page
  of the Bianchi modular forms database}}}.
\newblock
  \mbox{\url{http://www.lmfdb.org/ModularForm/GL2/ImaginaryQuadratic/}}, 2013.
\newblock [Online].

\bibitem{Loeffler:2010ab}
D.~Loeffler.
\newblock Density of classical points in eigenvarieties.
\newblock {\em Math. Res. Lett.}, 18(5):983--990, 2011.

\bibitem{Loeffler:2010aa}
D.~Loeffler and J.~Weinstein.
\newblock On the computation of local components of a newform.
\newblock {\em Mathematics of Computation 81 (2012), 1179-1200}, 08 2010.

\bibitem{MR0299559}
T.~Miyake.
\newblock On automorphic forms on {${\rm GL}_{2}$} and {H}ecke operators.
\newblock {\em Ann. of Math. (2)}, 94:174--189, 1971.

\bibitem{Mohamed:thesis}
A.~Mohamed.
\newblock {\em Some explicit aspects of modular forms over imaginary quadratic
  fields}.
\newblock PhD thesis, Universit\"at Duisburg-Essen, Oct 2011.

\bibitem{MR3235551}
K.~Nakamura.
\newblock Zariski density of crystalline representations for any {$p$}-adic
  field.
\newblock {\em J. Math. Sci. Univ. Tokyo}, 21(1):79--127, 2014.

\bibitem{MR2738153}
J.~Quer.
\newblock Dimensions of spaces of modular forms for {$\Gamma_H(N)$}.
\newblock {\em Acta Arith.}, 145(4):373--395, 2010.

\bibitem{MR3091734}
A.~D. Rahm and M.~H. {\c S}eng{\"u}n.
\newblock On level one cuspidal {B}ianchi modular forms.
\newblock {\em LMS J. Comput. Math.}, 16:187--199, 2013.

\bibitem{Rahm:2017aa}
A.~D. Rahm and P.~Tsaknias.
\newblock Genuine {B}ianchi modular forms of higher level at varying weight and
  discriminant.
\newblock {\em J. Th\'{e}or. Nombres Bordeaux}, 31(1):27--48, 2019.

\bibitem{MR728453}
J.~Schwermer and K.~Vogtmann.
\newblock The integral homology of {${\rm SL}_{2}$}\ and {${\rm PSL}_{2}$}\ of
  {E}uclidean imaginary quadratic integers.
\newblock {\em Comment. Math. Helv.}, 58(4):573--598, 1983.

\bibitem{MR2859903}
M.~H. {\c S}eng{\"u}n.
\newblock On the integral cohomology of {B}ianchi groups.
\newblock {\em Exp. Math.}, 20(4):487--505, 2011.

\bibitem{Serban:2016aa}
V.~{Serban}.
\newblock {An Infinitesimal $p$-adic Multiplicative Manin-Mumford Conjecture}.
\newblock {\em J. Th\'{e}or. Nombres Bordeaux}, 30(2):393--408, 2018.

\bibitem{MR2636500}
R.~L. Taylor.
\newblock {\em On congruences between modular forms}.
\newblock ProQuest LLC, Ann Arbor, MI, 1988.
\newblock Thesis (Ph.D.)--Princeton University.

\end{thebibliography}

\end{document}